\documentclass[12pt,twoside]{amsart}
\usepackage[margin=3cm]{geometry}
\usepackage[colorlinks=false]{hyperref}
\usepackage[english]{babel}
\usepackage{graphicx}
\usepackage{float}
\usepackage{amsmath,amsfonts,amssymb,amsthm}
\usepackage{lipsum}
\usepackage[T1]{fontenc}
\usepackage{fourier}
\usepackage{color}
\usepackage{esint}
\usepackage{caption}
\usepackage{siunitx}
\usepackage{ccicons}
\usepackage{subcaption}
\makeatletter
\def\blfootnote{\xdef\@thefnmark{}\@footnotetext}
\makeatother

\newcommand\ccnote{
    \blfootnote{The second author is supported by  NSFC No.12301077.}
}
\allowdisplaybreaks
\usepackage[export]{adjustbox}
\numberwithin{equation}{section}
\usepackage{setspace}
\renewcommand{\le}{\leqslant}
\renewcommand{\leq}{\leqslant}
\renewcommand{\ge}{\geqslant}
\renewcommand{\geq}{\geqslant}
\renewcommand{\mathbb}{\varmathbb}
\usepackage{fancyhdr}
\newtheorem{theorem}{Theorem}[section]
\newtheorem{lemma}[theorem]{Lemma}
\newtheorem{corollary}[theorem]{Corollary}
\newtheorem{proposition}[theorem]{Proposition}
\newtheorem{definition}[theorem]{Definition}
\newtheorem{remark}[theorem]{Remark}

\usepackage{comment}
\usepackage{enumerate}
\usepackage{hyperref}
\usepackage[alphabetic, msc-links, backrefs]{amsrefs}

\newcommand{\spt}{\operatorname{spt}}

\usepackage{amsthm}

\allowdisplaybreaks

\input epsf
\def\begfig {
\begin{figure}
\small }
\def\endfig {
\normalsize
\end{figure}
}

\begin{document}
\thispagestyle{empty}

\ccnote

\title[]{Quantitative Stability of the Clifford Torus as a Willmore Minimizer}

\author{Yuchen Bi}
\author{Jie Zhou}

\address{Yuchen Bi\\Mathematical Institute\\ Department of Pure Mathematics\\ University of Freiburg\\ Ernst-Zermelo-Stra{\ss}e 1 Freiburg im Breisgau \\ D-79104\\ Germany}
\email{yuchen.bi@math.uni-freiburg.de}
\address{Jie Zhou\\ School of Mathematical Sciences\\ Capital Normal University\\ 105 West Third Ring Road North\\ Haidian District\\ Beijing\\ 100048 P.R. China }
\email{zhoujiemath@cnu.edu.cn}
\maketitle

\noindent \textbf{Abstract.} \textit{For an integral $2$-varifold $V\subset \mathbb{S}^3$ with square-integrable mean curvature, unit density, and support of genus at least $1$, assume that its Willmore energy satisfies
\[
\mathcal{W}(V)\le 2\pi^2+\delta^2,\qquad \delta<\delta_0\ll1.
\]
We show that the support $\Sigma=\operatorname{spt}V$ is, after applying a suitable conformal transformation of $\mathbb{S}^3$, quantitatively close to the Clifford torus. More precisely, under an appropriate conformal normalization, the surface $\Sigma$ admits a $W^{2,2}$ conformal parametrization by the flat torus whose conformal factor and metric coefficients differ from those of the Clifford torus by at most $C\delta$.
}
\vskip0.3cm

\vspace{0.5cm}



\tableofcontents

\section{Introduction}

For an immersed surface $f:\Sigma\to \mathbb{S}^3$ in the round three-sphere, the Willmore energy is defined by
\[
\mathcal{W}(\Sigma)
=\int_{\Sigma}\!\left(1+\tfrac{1}{4}|H|^2\right)d\mu_f,
\]
where $H$ is the mean curvature and $d\mu_f$ is the area element of $f^{*}g_{\mathbb{S}^3}$.  
A fundamental property of $\mathcal{W}$ is its conformal invariance.  
By the theorem of Marques and Neves resolving the Willmore conjecture~\cite{MN14}, any closed oriented surface $\Sigma\subset\mathbb{S}^3$ with genus at least one satisfies
\[
\mathcal{W}(\Sigma)\ge 2\pi^2,
\]
with equality if and only if $\Sigma$ is, up to a conformal transformation, the Clifford torus
\[
\mathbb{T}^2=\mathbb{S}^1\!\left(\tfrac{1}{\sqrt2}\right)\times
\mathbb{S}^1\!\left(\tfrac{1}{\sqrt2}\right)\subset\mathbb{S}^3.
\]

The purpose of this paper is to study the quantitative rigidity of this theorem.  
Given $\delta<\delta_0\ll1$ and a surface $\Sigma$ with
\[
\mathcal{W}(\Sigma)\le 2\pi^2+\delta^2,
\]
we ask how close $\Sigma$ must be to a conformal image of the Clifford torus, in a sense that simultaneously controls the $W^{2,2}$-parametrization, the $L^\infty$-bound of the conformal factor, and the conformal structure.  
This parallels the classical quantitative rigidity for the genus zero case, where De Lellis-M\"uller~\cite{dLMu,dLMu2} proved optimal linear estimates for nearly umbilical surfaces, later extended to higher codimension by Lamm-Sch\"atzle~\cite{LS-2014} and to varifolds by Bi-Zhou~\cite{BZ25}.

A first step toward our result is to understand the topology of surfaces whose Willmore energy is close to the Clifford value.  
Using existing existence and energy asymptotic results, one obtains that if a surface $\Sigma$ satisfies $g(\Sigma)\ge1$ and $\mathcal{W}(\Sigma)$ is sufficiently close to $2\pi^2$, then in fact $g(\Sigma)=1$.  
This follows directly from the existence theory for Willmore minimizers (Simon for genus 1~\cite{LS93} and Bauer-Kuwert for genus $\ge2$~\cite{BK03}), together with the large-genus limit of Kuwert-Li-Sch\"atzle~\cite{KLS10} and the Willmore conjecture~\cite{MN14}.

Once the genus has been identified, Leon Simon's results~\cite{LS93} imply that any sequence of tori 
\(\Sigma_k \subset \mathbb{S}^3\) with \(\mathcal{W}(\Sigma_k)\to 2\pi^2\) admits, after suitable conformal reparametrization, 
a subsequence converging as varifolds to a limit surface whose genus is at least one.  
By lower semicontinuity of the Willmore energy, this limit has Willmore energy at most \(2\pi^2\).  
In view of the Willmore conjecture, the only possibility is that the limit is itself a torus of Willmore energy \(2\pi^2\).
It is well known to experts (see for instance Kuwert-Li~\cite{KL-2012},  Sch\"atzle~\cite{Sch13} and Rivi\'ere\cite{R14}) that, after parametrizing by a normalized flat torus metric, such a sequence is then qualitatively close to the Clifford torus. Section~2 contains a detailed proof of this qualitative convergence. We in fact formulate and prove a slightly more general statement: the model
surface may be any fixed genus~$\ge1$ minimal surface (not only the Clifford
torus), provided that no energy concentration occurs---in the precise sense
that the mean curvature has small $L^{2}$-norm and the area ratio is close to~$1$.
Our approach proceeds by first establishing convergence of the surfaces
\emph{as subsets of $\mathbb{S}^{3}$} under these assumptions, and only
afterwards upgrading the convergence to that of the associated conformal
parametrizations.  This is somewhat different in spirit from earlier
treatments, which work directly with compactness of $W^{2,2}$ conformal
immersions into Euclidean space.

The main difficulty is to promote this qualitative convergence to a quantitative, in particular linear, stability estimate.

A natural idea would be to expand the Willmore functional around the Clifford torus: one considers the second variation together with the higher-order remainder, analyzes the spectral properties of the resulting symmetric quadratic form, and then derives quantitative stability from the coercivity of this linearized operator.  However, in order for this approach to be effective, the remainder must be genuinely dominated by the quadratic term, which requires the surface to be $C^1$ or at least Lipschitz close to the Clifford torus.  Establishing such a priori closeness is delicate in our weak $W^{2,2}$-conformal setting.

For this reason, following the strategy of Marques-Neves, we avoid a direct perturbative analysis of the Willmore functional.  Instead, using the Heintze-Karcher-type inequality to relate the Willmore energy to the area of the canonical family $\Sigma_{(v,t)}$, we reformulate the problem in Section~3 as a minimal-surface stability problem.

To describe our approach, we recall the canonical 5-parameter family of Marques-Neves~\cite{MN14}.  
For $v\in\mathring{B}^4$, let
\[
F_v(x)=(1-|v|^2)\frac{x-v}{|x-v|^2}-v
\]
be the corresponding conformal transformation.  
If $\mathbb{S}^3\setminus\Sigma=A\cup A^*$ is the decomposition into connected components, we denote $\Sigma_v:=F_v(\Sigma)$ and $\mathbb{S}^3\setminus\Sigma_v=A_v\cup A_v^*$ with $A_v=F_v(A)$.  
Let $d_v$ be the signed distance to $\Sigma_v$, and define
\[
\Sigma_{(v,t)}=\partial\{x\mid d_v(x)<t\}, \qquad (v,t)\in\mathring{B}^4\times[-\pi,\pi].
\]
The Heintze-Karcher-type inequality~\cite{MN14,HK78,R99} gives
\[
\mathrm{Area}(\Sigma_{(v,t)})\le \mathcal{W}(\Sigma),
\]
with equality only at $t=0$ when $\Sigma_v$ is minimal.  

Marques and Neves further established the following fundamental result, often referred to as the \emph{$2\pi^2$-Theorem}:

\begin{theorem}[$2\pi^2$-Theorem, Marques-Neves~\cite{MN14}]\label{2pi2-theorem}
Let $\Sigma\subset\mathbb{S}^3$ be a smooth embedded surface of genus at least $1$.  
Then its canonical five-parameter family 
\[
\{\Sigma_{(v,t)}\}_{(v,t)\in \mathring{B}^4\times[-\pi,\pi]}
\]
satisfies
\[
\max_{(v,t)\in \mathring{B}^4\times[-\pi,\pi]}
\mathrm{Area}\bigl(\Sigma_{(v,t)}\bigr)\;\ge\; 2\pi^2.
\]
Moreover, equality holds if and only if $\Sigma$ is the Clifford torus.
\end{theorem}

With the five-parameter family $\{\Sigma_{(v,t)}\}$ of
Marques-Neves recalled above, we now indicate how it enters our proof.

The key point is to show that if a genus--one surface $\Sigma$ satisfies
$\mathcal{W}(\Sigma)\le 2\pi^2+\delta^2$, then a \emph{conformal}
transformation of $\mathbb{S}^3$ yields a surface $\Sigma'$ with
\[
\mathrm{Area}(\Sigma')\;\ge\;2\pi^2 - C\delta^2.
\]

The $2\pi^2$ theorem guarantees that the full $(v,t)$-family contains a slice
whose area is at least $2\pi^2$.  Since the variables $v$ encode conformal
transformations and $t$ is the signed distance parameter, our aim is to
realize such a high-area slice \emph{within the conformal}
four-parameter subfamily.  
For this, it is necessary to control the distance parameter $t$ in terms of
the energy gap~$\delta$, ensuring that the maximizing slice must lie close to
the $t=0$ slice.

A closer examination of the Heintze-Karcher inequality used by
Marques-Neves shows that, \emph{once the five-parameter family is built
from the quantitatively normalized surface obtained in Section~2}, the
problem reduces to proving that the maximizing parameter $(v_0,t_0)$ stays a
definite distance away from the boundary of the conformal group.  
Equivalently, there exists a universal constant $\eta_1>0$ such that
\[
|v_0|\;\le\;1-\eta_1 .
\]
This reduction is carried out in Section~3, where the technical
proof of Proposition~\ref{away from boundary of conformal group} is deferred to the Appendix.

To clarify why such a uniform bound is true, we now outline the geometric
intuition behind it.  When the conformal parameter $v$ approaches the boundary of
$\mathring{B}^4$, the transformed surface $\Sigma_{(v,t)}$ becomes
quantitatively close to a geodesic sphere, which forces its
area to stay strictly below $2\pi^2$.  
Thus such $v$ cannot maximize the canonical family.

More precisely, when $t=0$, the qualitative stability implies that $\Sigma_v= \Sigma_{(v,0)}$ is Hausdorff close to a geodesic sphere $S$, and the region enclosed by $\Sigma_v$ is similarly close to that enclosed by~$S$.  
A direct computation then shows that
\[
\mathrm{Area}(\Sigma_v)
\]
is close to $\mathrm{Area}(S)$, which is at most $4\pi$. Since $\Sigma_{(v,t)}$ is obtained from $\Sigma_v$ by shifting along its normal directions, the area difference satisfies an estimate of the form
\[
\bigl|\mathrm{Area}(\Sigma_{(v,t)})-\mathrm{Area}(\Sigma_v)\bigr|
    \;\lesssim\; |t|\,\mathcal{W}(\Sigma_v).
\]
Hence, for sufficiently small $|t|$, the area remains below~$5\pi$.

For larger $|t|$, we distinguish two cases.  
If $\Sigma_v$ stays a definite distance from the poles, we may apply the
Euclidean model: a region trapped between two nearly parallel half–spaces of
separation $\varepsilon\ll1$ has the property that, for any fixed
$t_0\sim1$, the boundaries of its $t$-parallel sets ($t\ge t_0$) become Lipschitz graphs over one
bounding plane with Lipschitz constant $\to0$ as $\varepsilon/t_0\to0$.  
Transferring this to the spherical setting yields that $\Sigma_{(v,t)}$ is a
small Lipschitz graph over a geodesic sphere and hence has area $<5\pi$.

If instead $\Sigma_v$ lies very close to a pole, its diameter is already
tiny, and a simple comparison–geometry argument shows that
$\Sigma_{(v,t)}$ remains of area $<5\pi$.

Thus in all cases $\mathrm{Area}(\Sigma_{(v,t)})<5\pi$. Since $5\pi < 2\pi^2$, no parameter $(v,t)$ with $|v|$ sufficiently close to $1$ can maximize the canonical area.  
This yields the desired estimate
\[
|v_0| \le 1 - \eta_1.
\]

With this uniform control of the conformal parameter, the remaining task is to
derive a quantitative stability estimate for minimal surfaces.  
Assembling the arguments developed so far leads to our main quantitative
rigidity theorem:

\begin{theorem}\label{Thm: main}
Let $V=\underline{v}(\Sigma,1)$ be an integral $2$-varifold in $\mathbb{S}^3$
with unit density and mean curvature $H\in L^2(d\mu)$, and assume its support
$\Sigma$ has genus at least $1$.  
If the Willmore energy
\[
\mathcal{W}(V)=\int_\Sigma\Bigl(1+\tfrac14 |H|^2\Bigr)\,d\mu
\]
satisfies
\[
\mathcal{W}(V)\le 2\pi^2+\delta^2,
\qquad \delta>0 \text{ sufficiently small},
\]
then, after a suitable conformal transformation of $\mathbb{S}^3$, there exists
a homeomorphism
\[
f:\mathbb{S}^1\times\mathbb{S}^1 \longrightarrow \Sigma \subset\mathbb{S}^3
\]
such that in the standard coordinates $(\theta,\varphi)$ on
$\mathbb{S}^1\times\mathbb{S}^1$ the pull-back metric takes the form
\[
f^{*}g_{\mathbb{S}^3}
= e^{2u}\,(a\, d\theta^2 + 2 b\, d\theta\, d\varphi + c\, d\varphi^2).
\]

Let $f_0:\mathbb{S}^1\times\mathbb{S}^1\to\mathbb{S}^3$ denote the Clifford torus
embedding
\[
f_0(\theta,\varphi)
= \frac{1}{\sqrt{2}}(\cos\theta,\;\sin\theta,\;\cos\varphi,\;\sin\varphi),
\]
and write its image as $\mathbb{T}^2:=f_0(\mathbb{S}^1\times\mathbb{S}^1)$.  
Then
\[
f_0^{*} g_{\mathbb{S}^3}
= \tfrac12 (d\theta^2+d\varphi^2).
\]

The following quantitative rigidity estimates hold:
\begin{enumerate}

\item[\rm(1)]
\[
\| f - f_0 \|_{W^{2,2}(\mathbb{S}^1\times\mathbb{S}^1)} \le C\,\delta.
\]

\item[\rm(2)]
\[
\|u\|_{L^\infty(\mathbb{S}^1\times\mathbb{S}^1)} \le C\,\delta.
\]

\item[\rm(3)]
\[
|a-\tfrac12| + |b| + |c-\tfrac12| \le C\,\delta.
\]

\end{enumerate}
Here $C$ is a universal constant independent of $\Sigma$.
\end{theorem}

Recent work of Rupp-Scharrer~\cite{RS25}, making use of the regularity theorem
developed in~\cite{BZ-2022b}, establishes density of smooth surfaces among
integral $2$-varifolds with square-integrable second fundamental form and
Willmore energy below $8\pi$.  
Since unit density together with $H\in L^2$ already yields square-integrable
second fundamental form by~\cite{BZ-2022b}, their result is closely related to
the class of varifolds considered here.

In the appendix we prove a density statement suited to our setting:
compactly supported integral $2$-varifolds with unit density and
square-integrable mean curvature can be approximated in $W^{2,2}$ by smooth
embedded surfaces.  
Our construction provides global bilipschitz control between the limiting
parametrization $F$ and its smoothings $F_\eta$, which ensures
embeddedness of the approximating surfaces without requiring any additional
smallness assumptions on the Willmore energy.

\medskip
\noindent\textbf{Outline of the paper.}
In Section~2, we establish the topological stability and qualitative geometric stability for surfaces with Willmore energy sufficiently close to $2\pi^2$.  
In particular, we show that such a surface must have genus one and is, after suitable conformal normalization, qualitatively close to the Clifford torus.

In Section~3, we analyze the Marques-Neves canonical family and prove the key estimate that under the assumption $\mathcal{W}(\Sigma)\le 2\pi^2+\delta^2$, one can find a conformal transformation of $\mathbb{S}^3$ such that the transformed surface $\Sigma'$ satisfies
\[
\mathrm{Area}(\Sigma')\ge 2\pi^2 - C\delta^2,
\]
for a universal constant $C>0$.  
This reduces the Willmore stability problem to a quantitative stability statement for minimal surfaces.

Section~4 contains the linearization and stability analysis.  
Using the estimate from Section~3, we derive the quantitative bounds on the conformal factor, the conformal structure, and the $W^{2,2}$-distance to the Clifford torus, thereby proving Theorem~\ref{Thm: main}.

Finally, in the Appendix we provide the proof of Proposition~\ref{away from boundary of conformal group} and we also establish the density and parametrization results for compactly
supported integral $2$-varifolds with unit density and $L^2$-mean curvature,
showing in particular that such varifolds admit global conformal
parametrizations and can be approximated by smooth embedded surfaces.

\section{Qualitative Preliminaries for Quantitative Rigidity}

In this section we develop the qualitative ingredients needed for the quantitative rigidity theorem.  
We begin with a topological stability statement: any surface in $\mathbb{S}^3$ whose Willmore energy is sufficiently close to $2\pi^2$ and whose genus is at least one must in fact be an embedded torus.

\begin{lemma}\label{topological rigidity} 
Assume $\delta<\delta_0\ll 1$ and $\Sigma$ is a closed immersed surface in 
$\mathbb{S}^3$ with genus at least one and 
$\mathcal{W}(\Sigma)\le 2\pi^2+\delta^2$. 
Then $\Sigma$ is an embedded torus.  
\end{lemma}

\begin{proof}
Regarding $\Sigma\subset\mathbb{S}^3\subset\mathbb{R}^4$ as an immersed 
surface in $\mathbb{R}^4$, its Willmore energy equals 
\begin{align*}
\mathcal{W}(\Sigma)
= \frac{1}{4}\int_{\Sigma}|\vec H|^2\,d\mu,
\end{align*}
where $\vec H$ is the mean curvature of $\Sigma\subset\mathbb{R}^4$.  
By the monotonicity formula, for any $p\in f(\Sigma)$,
\begin{align*}
\Theta(\Sigma,p)
= \lim_{r\to0}\frac{\mu(f^{-1}(B_r(p)))}{\pi r^2}
\le \frac{1}{16\pi}\int_{\Sigma}|\vec H|^2\,d\mu
\le \frac{2\pi^2+\delta^2}{4\pi}
<2.
\end{align*}
Since $\Sigma$ is immersed, $\Theta(\Sigma,p)$ is an integer and hence equals 
one; in particular $f:\Sigma\to\mathbb{S}^3$ is an embedding, and $\Sigma$ is 
oriented.

By the existence result of \cite{BK03} and the large-genus limit 
\cite{KLS10}, for each $g\ge1$ there exists a smooth embedded closed surface 
$\Sigma_g\subset\mathbb{S}^3$ with genus $g$ and 
\begin{align*}
\beta_g^3
:= \inf\{\,\mathcal{W}(\Sigma)\mid 
\Sigma\subset\mathbb{S}^3 \text{ with genus } g \,\}
= \mathcal{W}(\Sigma_g),
\end{align*}
such that $\lim_{g\to\infty}\beta_g^3=8\pi>2\pi^2$.
Therefore, there exists $N\ge1$ such that for any $g\ge N$, 
$\beta_g^3>2\pi^2+\delta_0^2$.  
The rigidity part of the Marques-Neves theorem shows 
$\beta_g^3>2\pi^2$ for every $g>1$.  
Now choose $\delta_0>0$ such that 
\[
\delta_0^2<\delta_1^2
:=\min_{2\le g\le N}(\beta_g^3-2\pi^2).
\]
Then for any $\delta<\delta_0$ and any 
$\Sigma\subset\mathbb{S}^3$ with $\mathcal{W}(\Sigma)\le2\pi^2+\delta^2$, we 
have $\mathcal{W}(\Sigma)<\beta_g^3$ for all $g\ge2$.  
Hence $\mathrm{genus}(\Sigma)=1$, and therefore $\Sigma$ is an embedded torus 
in $\mathbb{S}^3$.
\end{proof}

With the topological rigidity , we are ready to establish the following local estimate.

\begin{proposition}\label{basic setting}
Assume $\Sigma_i\subset\mathbb{S}^3$ is a sequence of embedded tori with
\[
\mathcal{W}(\Sigma_i)= 2\pi^2+\delta_i^2
\qquad\text{and}\qquad
\lim_{i\to\infty}\delta_i=0.
\]
Then there exist conformal transformations $F_i:\mathbb{S}^3\to\mathbb{S}^3$ such that the following holds.

For any fixed $\gamma>0$, there exist $r_0=r_0(\gamma)>0$ and $I_0=I_0(\gamma)\in\mathbb{N}$ such that for any $i\ge I_0$, any $x\in F_i(\Sigma_i)$ and any $r<r_0$, there exist a conformal parametrization
\[
\varphi :D_{r}\to F_i(\Sigma_i)
\]
and a $2$-dimensional affine plane $T_{x,r}$ passing through $x$ such that
\begin{align*}
F_i(\Sigma_i)\cap B_r(x)\subset \varphi(D_r)\subset F_i(\Sigma_i)\cap B_{(1+\gamma)r}(x),
\end{align*}
\begin{align*}
|\varphi(\tau)-\tau|\le \gamma r \quad \text{for all }\tau\in D_r,
\end{align*}
and
\begin{align*}
(1-\gamma)|\tau_1-\tau_2|\le |\varphi(\tau_1)-\varphi(\tau_2)|\le (1+\gamma)|\tau_1-\tau_2|
\quad \text{for all }\tau_1,\tau_2\in D_r.
\end{align*}
\end{proposition}
\begin{proof}
By Lemma~\ref{topological rigidity}, each $\Sigma_i$ (and hence each $\widetilde{\Sigma}_i$) is an embedded torus for $i$ sufficiently large.  
Having established the topological type, we may then apply Simon's work \cite{LS93}: after composing with suitable conformal transformations $F_i$, the surfaces
\[
\widetilde{\Sigma}_i := F_i(\Sigma_i)
\]
converge in the Hausdorff sense and as varifolds to a smooth surface $\Sigma$ whose genus is at least $1$, and moreover
\[
\mathcal{W}(\Sigma)\le \liminf_{i\to\infty}\mathcal{W}(\widetilde{\Sigma}_i)=2\pi^2.
\]
By the rigidity part of the Willmore conjecture \cite{MN14}, the limit $\Sigma$ must be the Clifford torus.  
After composing with a fixed conformal automorphism of $\mathbb{S}^3$ if necessary, we may assume that
\[
\Sigma=\mathbb{T}^2.
\]

Since
\[
\mathcal{W}(\widetilde{\Sigma}_i)
   = \int_{\widetilde{\Sigma}_i}\Bigl(1+\tfrac14 |H_i|^2\Bigr)\, d\mathcal{H}^2
   = 2\pi^2+\delta_i^2,
\]
and $\mathcal{H}^2(\widetilde{\Sigma}_i)\to\mathcal{H}^2(\mathbb{T}^2)=2\pi^2$ by varifold convergence, it follows that
\[
\int_{\widetilde{\Sigma}_i}|H_i|^2\, d\mathcal{H}^2 \longrightarrow 0.
\]

Thus $\widetilde{\Sigma}_i\to\mathbb{T}^2$ in Hausdorff distance, with $L^2$-mean curvature tending to zero.  
Applying the regularity theorem of Bi-Zhou \cite[Thm.~1.1]{BZ-2022b}, the proposition follows.

\end{proof}

  In this paper we will use the above local estimate repeatedly.  For this purpose we introduce the following notion.
\begin{definition}
Let $\gamma>0$ and $r_0>0$.  
For a surface $\Sigma\subset\mathbb{R}^n$ and a point $x\in\Sigma$,  
we say that $x$ is a \emph{$(\gamma,r_0)$-regular point} if for every $r<r_0$  
there exist an affine $2$-plane $T_{x,r}\subset\mathbb{R}^n$ passing through $x$  
and a conformal parametrization
\[
\varphi:=\varphi_{x,r}:D_r(x):=T_{x,r}\cap B_r(x)\longrightarrow\Sigma
\]
such that $\varphi(x)=x$, 
\begin{equation}\label{eq:containment}
\Sigma\cap B_r(x)\subset\varphi(D_r(x))\subset\Sigma\cap B_{(1+\gamma)r}(x),
\end{equation}
\begin{equation}\label{eq:graph-closeness}
|\varphi(\tau)-\tau|\le\gamma r,\qquad\forall\,\tau\in D_r(x),
\end{equation}
and the bi-Lipschitz bound
\begin{equation}\label{eq:bilip}
(1-\gamma)|\tau_1-\tau_2|\le|\varphi(\tau_1)-\varphi(\tau_2)|
\le(1+\gamma)|\tau_1-\tau_2|,\quad\forall\,\tau_1,\tau_2\in D_r(x).
\end{equation}
Such a map $\varphi$ is called a \emph{$(\gamma,r)$-regular conformal parametrization}.

Note that the domain $D_r(x)$ lies in the affine plane $T_{x,r}$.  
Fixing an isometry $\sigma:D_r\to D_r(x)$ with $\sigma(0)=x$,  
we may identify $D_r(x)$ with the standard Euclidean disk $D_r\subset\mathbb{R}^2$.  
The composition
\[
\tilde{\varphi}=\varphi\circ\sigma:D_r\longrightarrow\Sigma
\]
is again called a \emph{$(\gamma,r)$-regular conformal parametrization}.  
It satisfies $\tilde{\varphi}(0)=x$, the inclusions in \eqref{eq:containment},  
and
\[
(1-\gamma)|\tau_1-\tau_2|\le|\tilde{\varphi}(\tau_1)-\tilde{\varphi}(\tau_2)|
\le(1+\gamma)|\tau_1-\tau_2|,\quad\forall\,\tau_1,\tau_2\in D_r.
\]
We say that the surface $\Sigma$ is \emph{$(\gamma,r_0)$-regular} if every $x\in\Sigma$ is $(\gamma,r_0)$-regular.
\end{definition}

Having introduced the notion of $(\gamma,r_0)$-regularity, we proceed to prove the following local convergence lemma.

\begin{lemma}[Sequential local convergence]\label{local estimate}
Let $\gamma,r_0,V\in(0,\infty)$ be fixed.  
Let $\Sigma_0\subset\mathbb{S}^n$ be a minimal surface and 
$\Sigma_i\subset\mathbb{S}^n\subset\mathbb{R}^{n+1}$  
a sequence of $(\gamma,r_0)$-regular surfaces satisfying  
\[
\mathcal{H}^2(\Sigma_i)\le V,\qquad 
d_{\mathcal H}(\Sigma_i,\Sigma_0)+\|H_i\|_{L^2(\Sigma_i)}\to0 .
\]
For any $x_i\in\Sigma_i$ with $x_i\to x\in\Sigma_0$, let  
\[
\tilde\varphi_i:D_{r_0}\to\Sigma_i,\qquad 
\tilde\varphi_i(0)=x_i,
\]
be $(\gamma,r_0)$-regular conformal parametrizations.  
Then, after passing to a subsequence, there exists a $(\gamma,r_0)$-regular conformal parametrization  
\[
\tilde\varphi:D_{r_0}\to\Sigma_0,\qquad \tilde\varphi(0)=x,
\]
such that for every $\sigma<r_0$,
\[
\|\tilde\varphi_i-\tilde\varphi\|_{W^{2,2}(D_\sigma)}
+\|u_i-u\|_{L^\infty(D_\sigma)}\longrightarrow 0 ,
\]
where $\tilde\varphi_i^*g_{\mathbb{R}^{n+1}}=e^{2u_i}g_{\mathbb R^2}$ and  
$\tilde\varphi^*g_{\mathbb{R}^{n+1}}=e^{2u}g_{\mathbb R^2}$.
\end{lemma}

\begin{proof}
Since each $\tilde\varphi_i$ is $(\gamma,r_0)$-regular,  
\[
1-\gamma\le e^{u_i}\le 1+\gamma,\qquad
\Delta\tilde\varphi_i=e^{2u_i}(H_i-2\tilde\varphi_i).
\]
As $\|H_i\|_{L^2(\Sigma_i)}\to0$, we obtain  
$\|e^{u_i}H_i\|_{L^2(D_{r_0})}\to0$.  
Uniform $(\gamma,r_0)$-regularity yields uniform $W^{2,2}$–bounds; hence,
\[
\tilde\varphi_i\rightharpoonup\tilde\varphi \text{ in }W^{2,2}(D_{r_0}),\qquad 
\tilde\varphi_i\to\tilde\varphi \text{ in }C^0_{\mathrm{loc}}(D_{r_0}).
\]
The Hausdorff convergence implies $\tilde\varphi(D_{r_0})\subset\Sigma_0$,  
and the bilipschitz bounds pass to the limit, so $\tilde\varphi$ is  
$(\gamma,r_0)$-regular and conformal.

Since $\Sigma_0$ is minimal in $\mathbb{S}^n$, the limit map satisfies  
\[
\Delta\tilde\varphi=-|D\tilde\varphi|^2\,\tilde\varphi.
\]
Interior elliptic regularity implies  
\begin{equation}\label{eq:Hess-bound-smooth}
\|D^2\tilde\varphi\|_{L^\infty(D_\sigma)}\le C(\sigma)
\quad\forall\,\sigma<r_0.
\end{equation}

To compare conformal factors, introduce  
\[
e_j=e^{-u}\partial_j\tilde\varphi,\qquad 
e_j^i=e^{-u_i}\partial_j\tilde\varphi_i.
\]
These satisfy
\[
\Delta u=-\ast(de_1\wedge de_2),\qquad 
\Delta u_i=-\ast(de_1^i\wedge de_2^i),
\]
hence
\[
\Delta(u-u_i)
=-\ast\!\left[d(e_1-e_1^i)\wedge de_2+de_1^i\wedge d(e_2-e_2^i)\right].
\]
Using $1-\gamma\le e^u,e^{u_i}\le1+\gamma$ and \eqref{eq:Hess-bound-smooth},
\[
\|e_j-e_j^i\|_{W^{1,2}(D_\sigma)}
\le C\|\tilde\varphi_i-\tilde\varphi\|_{W^{2,2}(D_\sigma)}.
\]
Thus by Wente's inequality\cite{W69},
\[
\|u_i-u\|_{L^\infty(D_\sigma)}
\le C\|\tilde\varphi_i-\tilde\varphi\|_{W^{2,2}(D_\sigma)}.
\]

Subtracting the equations for $\tilde\varphi_i$ and $\tilde\varphi$,  
using $\|H_i\|_{L^2(\Sigma_i)}\to0$ and elliptic estimates, we obtain  
\[
\|\tilde\varphi_i-\tilde\varphi\|_{W^{2,2}(D_\sigma)}\to0.
\]
The convergence of $u_i$ follows.  
\end{proof}

\begin{remark}
The assumption that $\Sigma_0$ is minimal in $\mathbb{S}^n$ is not essential
for Lemma~\ref{local estimate}, nor is the choice of the ambient sphere.  
The argument works in any fixed ambient manifold $M$, provided there is no
concentration of mean curvature along the sequence.

For instance, it suffices to assume that there exists a smooth vector field
$X$ on $M$ whose restriction to $\Sigma_0$ equals its mean curvature vector,
and that the surfaces $\Sigma_i$ satisfy
\[
\|H_i - X\|_{L^2(\Sigma_i)}\to0,
\]
together with uniform $(\gamma,r_0)$-regularity.  
Under these hypotheses all steps remain valid---the $W^{2,2}$ bounds, the
control of the conformal factors via Wente's inequality, and the strong
$W^{2,2}$ convergence $\tilde\varphi_i\to\tilde\varphi$.
The minimal case $X\equiv0$ is stated only because it is the one needed in our
applications (notably when $\Sigma_0$ is the Clifford torus).
\end{remark}

This local convergence result has two immediate consequences.  
The first is that the nearest-point projection from $\Sigma_i$ to $\Sigma_0$ has degree~$1$.

\begin{proposition}\label{prop:degree}
Under the assumptions of Lemma~\ref{local estimate},  
the nearest-point projection $\pi_i:\Sigma_i\to\Sigma_0$ has degree $1$  
for all sufficiently large $i$.
\end{proposition}

\begin{proof}
Since $d_{\mathcal H}(\Sigma_i,\Sigma_0)\to0$, the projection $\pi_i$ is well defined for $i$ large.  
Fix $x_i\to x$ and choose the parametrizations $\tilde\varphi_i,\tilde\varphi$  
from Lemma~\ref{local estimate}.  
For $r_0$ small,
\[
\tilde\varphi(D_{r_0/3})
\subset \pi_i\!\left(\tilde\varphi_i(D_{r_0/2})\right)
\subset \tilde\varphi(D_{2r_0/3}).
\]
Thus on $D_{r_0/2}$ we may write
\[
\pi_i\circ\tilde\varphi_i=\tilde\varphi\circ g_i,
\]
where $g_i\to\mathrm{Id}$ in $C^0$.  
In particular,
\[
g_i(D_{r_0/2})\supset D_{r_0/3},\qquad 
g_i^{-1}(D_{r_0/4})\subset D_{r_0/3}.
\]

Select a partition of unity $\{\eta_1,\eta_2\}$ subordinate to  
$\{D_{r_0/2}\setminus D_{r_0/4},\,D_{r_0/3}\}$ and define the homotopy  
\[
H_{i,t}(z)=\eta_1(z)\,g_i(z)+\eta_2(z)\,\bigl(z+t(g_i(z)-z)\bigr),\qquad t\in[0,1].
\]
For $i$ large, $0\notin H_{i,t}(\partial D_{r_0/2})$ for all $t$,  
and all preimages of $0$ lie in $D_{r_0/4}$, where  
$H_{i,t}(z)=z+t(g_i(z)-z)$.  
Since $0$ is a regular value of $H_{i,0}$ and $H_{i,0}^{-1}(0)=\{0\}$,  
\[
\deg(g_i)=\deg(H_{i,1})=\deg(H_{i,0})=1.
\]
Because $\tilde\varphi_i$ and $\tilde\varphi$ are orientation-preserving embeddings,  
\[
\deg(\pi_i)=\deg(g_i)=1.
\]
\end{proof}


A second immediate consequence of Lemma~\ref{local estimate} concerns the convergence of conformal maps whose targets are the surfaces $\Sigma_i$.

\begin{lemma}[Convergence of conformal maps]\label{lem:conformal-convergence}
Let $\gamma,r_0,V$ be as above.  
Let $\Sigma_i, \Sigma_0$ be $(\gamma,r_0)$-regular surfaces with  
\[
d_{\mathcal H}(\Sigma_i,\Sigma_0)+\|H_i\|_{L^2(\Sigma_i)}\to0 .
\]
For any $x_i\to x\in\Sigma_0$, let
\[
f_i:D_{r_0/2}\to\Sigma_i,\qquad f_i(0)=x_i,
\]
be conformal maps satisfying
\[
f_i^*g_{\mathbb R^{n+1}}=e^{2u_i}g_{\mathbb R^2},\qquad 
|\nabla f_i|\le1,\qquad |\nabla f_i(0)|=1.
\]
Then, after passing to a subsequence, there exists a nonconstant conformal map
\[
f:D_{r_0/2}\to\Sigma_0,\qquad f(0)=x,
\]
such that
\[
f_i\to f \ \text{in }W^{2,2}(D_{r_0/4}),\qquad 
u_i\to u\ \text{in }L^\infty(D_{r_0/4}),
\]
with $f^*g_{\mathbb R^{n+1}}=e^{2u}g_{\mathbb R^2}$.
\end{lemma}

\begin{proof}
By Lemma~\ref{local estimate} there exist parametrizations  
$\varphi_i:D_{r_0}\to\Sigma_i$ and a $(\gamma,r_0)$-regular  
$\varphi:D_{r_0}\to\Sigma_0$ such that
\[
\varphi_i\to\varphi \text{ in }W^{2,2}(D_\sigma),\qquad 
v_i\to v \text{ in }L^\infty(D_\sigma),
\quad\forall\,\sigma<r_0.
\]

Since $|\nabla f_i|\le1$ and $f_i(0)=x_i$,  
\[
f_i(D_{r_0/2})\subset B_{r_0/2}(x_i)\subset \varphi_i(D_{r_0}),
\]
and hence there exists a unique conformal embedding  
\[
\psi_i:D_{r_0/2}\to D_{r_0},\qquad 
f_i=\varphi_i\circ\psi_i,\qquad \psi_i(0)=0.
\]

Comparing conformal factors yields
\[
u_i=v_i\circ\psi_i+\log|\psi_i'|.
\]
From $|\nabla f_i(0)|=1$ and $1-\gamma\le e^{v_i}\le1+\gamma$ we obtain a uniform nondegenerate bound on $|\psi_i'(0)|$.  
Thus $\{\psi_i\}$ is a normal family, and by Montel’s theorem,
\[
\psi_i\to\psi
\quad\text{locally uniformly,}
\]
where $\psi$ is holomorphic, nonconstant, and satisfies $\psi(0)=0$.

Define $f=\varphi\circ\psi$.  
Since $\psi_i\to\psi$ in $C^\infty(D_{r_0/4})$ and  
$\varphi_i\to\varphi$ in $W^{2,2}(D_{r_0/4})$,  
\[
f_i=\varphi_i\circ\psi_i\to\varphi\circ\psi=f
\quad\text{in }W^{2,2}(D_{r_0/4}).
\]
The convergence of $u_i$ follows from  
$u_i=v_i\circ\psi_i+\log|\psi_i'|$ and the uniform convergence $v_i\to v$.  
\end{proof}

We can now prove a global convergence result for conformal parametrizations.  
The qualitative part of our stability theorem will be an immediate consequence of this statement.

\begin{theorem}[Global convergence of conformal parametrizations]\label{thm:global-convergence}
Fix $\gamma,r_0,V\in(0,\infty)$.  
Let $\Sigma_0\subset\mathbb{S}^n$ be a compact minimal surface and 
$\Sigma_i\subset\mathbb{S}^n\subset\mathbb{R}^{n+1}$  
a sequence of $(\gamma,r_0)$-regular compact surfaces with $\mathcal{H}^2(\Sigma_i)\le V$ and 
\[
d_{\mathcal H}(\Sigma_i,\Sigma_0)+\|H_i\|_{L^2(\Sigma_i)}\longrightarrow 0.
\]
Here $H_i$ denote the mean curvature vectors of $\Sigma_i$
in the unit sphere $\mathbb{S}^n\subset\mathbb{R}^{n+1}$.

Assume moreover that all $\Sigma_i$ and $\Sigma_0$ have the same genus $g\ge1$.  
For each $i$ let
\[
f_i:(M_i,h_i)\longrightarrow\Sigma_i
\]
be a conformal parametrization, where $M_i$ is a volume-one flat torus if $g=1$, and a compact hyperbolic surface of constant curvature $-1$ if $g\ge2$.

Then, after passing to a subsequence, there exists a compact Riemann surface $(M_0,h_0)$ of genus $g$ and a conformal parametrization
\[
f_0:(M_0,h_0)\longrightarrow\Sigma_0.
\]
such that  $(M_i,h_i)$ converge smoothly to $(M_0,h_0)$ in the Cheeger-Gromov sense. That is,  we may choose diffeomorphisms
\[
\psi_i : M_0 \longrightarrow M_i
\]
such that $\psi_i^{*}h_i \to h_0$ smoothly; in particular, for $i$ large, each $\psi_i$ is $(1+\varepsilon_i)$-bi-Lipschitz with $\varepsilon_i\to 0$.  
With respect to these diffeomorphisms, the following hold:
\begin{itemize}
\item $f_i\circ\psi_i \to f_0$ in $W^{2,2}$;
\item the corresponding conformal factors $u_i$, defined by $f_i^{*}g_{\mathbb{R}^{n+1}} = e^{2u_i}h_i$, satisfy
\[
u_i \circ \psi_i \longrightarrow u_0 \quad\text{in } L^\infty,
\]
where $f_0^{*}g_{\mathbb{R}^{n+1}} = e^{2u_0}h_0$.
\end{itemize}
\end{theorem}

\begin{proof}
Since $f_i:(M_i,h_i)\to(\Sigma_i,g_i)$ is a conformal parametrization, its Dirichlet energy is equal to the area of $\Sigma_i$, and hence
\[
E(f_i)=\int_{M_i}|df_i|^2\,d\mu_{h_i}\;= \;2\mathcal{H}^2(\Sigma_i)\le V.
\]
Thus the energies $E(f_i)$ are uniformly bounded.

We claim that $\mathrm{inj}(M_i)|df_i|$ (which is $(\mathrm{inj}(M_i)\operatorname{tr}_{h_i}(f_i^*g_i))^{1/2}$) are uniformly bounded in $L^\infty$.  
Assume not. Then there exists a sequence of points $p_i\in M_i$ such that
\[
\mathrm{inj}(M_i)|df_i(p_i)|=\max_{x\in M_i}\mathrm{inj}(M_i)|df_i(x)|\longrightarrow\infty.
\]
Rescale the domain metrics by $\hat h_i=\lambda_i^2 h_i$ where $\lambda_i:=|df_i(p_i)|$ and consider the maps
\[
\hat f_i:(M_i,\hat h_i)\to\Sigma_i,\qquad \hat f_i=f_i.
\]
In the metric $\hat h_i$ we have $|d\hat f_i(p_i)|_{\hat h_i}=1$.  
Since $h_i$ has constant curvature (zero or $-1$) and the rescaling is by a constant factor, the sectional curvatures of $\hat h_i$ are uniformly bounded on compact sets, and by using exponential coordinates at $p_i$ with respect to $\hat h_i$ we may identify larger and larger balls around $p_i$ with Euclidean disks in $\mathbb{R}^2$.  
On each fixed disk, the maps $\hat f_i$ are weakly conformal immersions with uniformly bounded energy (coming from the global energy bound and conformal invariance in two dimensions), and their targets $\Sigma_i$ converge to $\Sigma_0$ in Hausdorff sense.  
By Lemma~\ref{lem:conformal-convergence}, after passing to a subsequence,
\[
\hat f_i\to f_\infty:\mathbb{R}^2\to\Sigma_0
\]
on compact subsets, where $f_\infty$ is a weakly conformal harmonic map. The normalization $|d\hat f_i(p_i)|_{\hat h_i}=1$ and  locally uniformly convergence of $\hat{u}_i=\frac{1}{2}(\log |d\hat{f}_i|-\log 2)$ imply
\[
|df_\infty(0)|=1,
\]
so $f_\infty$ is nonconstant.  
Moreover, the energy bound passes to the limit and $f_\infty$ has finite energy.  
By the removable singularity theorem for harmonic maps, $f_\infty$ extends across the point at infinity to a nonconstant branched conformal harmonic map
\[
F:\mathbb{S}^2\to\Sigma_0.
\]
Since $\Sigma_0$ has genus $g\ge1$, this is impossible.  
Thus our assumption was false, and there exists a constant $C$ such that
\[
\mathrm{inj}(M_i)\|df_i\|_{L^\infty(M_i)}\le C\qquad\text{for all }i.
\]

We now show that the injectivity radii $\mathrm{inj}(M_i)$ are uniformly bounded from below.  
Suppose instead that $\mathrm{inj}(M_i)\to0$.  
Let $\gamma_i\subset M_i$ be a closed geodesic realizing  the injectivity radius, so that
\[
\ell(\gamma_i)= 2\,\mathrm{inj}(M_i)\longrightarrow0.
\]
When $g=1$, each $(M_i,h_i)$ is a flat torus of area one, and the existence of such a short closed geodesic implies that $(M_i,h_i)$ contains a long thin flat cylinder neighbourhood of $\gamma_i$.  
When $g\ge2$, each $(M_i,h_i)$ is a compact hyperbolic surface of constant curvature $-1$, and the collar lemma shows that a short simple closed geodesic again carries a long hyperbolic cylinder neighbourhood.  
In either case we obtain a family of embedded cylinders
\[
C_i\cong[-L_i,L_i]\times\mathbb{S}^1\subset M_i,
\]
with $L_i\to\infty$, such that the central loop $\{0\}\times\mathbb{S}^1$ is a (nearly) shortest closed geodesic of length $\ell(\gamma_i)\to0$ and the metric on $C_i$ is uniformly equivalent (after a constant rescaling) to the standard flat cylinder metric.  
Restricting $f_i$ to $C_i$ and using the uniform $L^\infty$ bound on $|df_i|$ and the global energy bound, we see that the energy of $f_i$ on each fixed subcylinder $[-\ell,\ell]\times\mathbb{S}^1$ is uniformly bounded, independently of $i$.  
Passing to a subsequence and using estimates for conformal maps as above, we obtain a weakly conformal harmonic map
\[
f_\infty:\mathbb{R}\times\mathbb{S}^1\to\Sigma_0
\]
of finite energy.  

We now show that $f_\infty$ is not constant. In fact, we claim that
\[
\liminf_{i\to \infty}\Bigl(\ell(\gamma_i)\sup_{t\in\gamma_i}|df_i(t)|\Bigr)>0.
\]
Otherwise, the length of $f_i(\gamma_i)$ tends to $0$, so the length of the curves
\[
\pi_i\circ f_i(\gamma_i)\subset\Sigma_0
\]
also tends to $0$, and $\pi_i\circ f_i(\gamma_i)$ is null-homotopic for $i$ large.  Since $\pi_i$ is a degree $1$ map by Proposition~\ref{prop:degree}, it follows that $\pi_i\circ f_i(\gamma_i)$ is null-homotopic in $\Sigma_0$, and thus $f_i(\gamma_i)$ 
is also null-homotopic. Here we use the fact that any degree $1$ self-map of a closed surface is a homotopy 
equivalence: this is clear when the genus is $\le 1$, and for orientable surfaces of 
genus $\ge 2$ it follows from the Hopfian property of surface groups. 
On the other hand, $f_i$ is a homeomorphism, so this contradicts the choice of $\gamma_i$ as a shortest nontrivial closed geodesic. By Lemma~\ref{lem:conformal-convergence}, we conclude that $f_\infty$ is nontrivial.

By the removable singularity theorem for harmonic maps, $f_\infty$ extends to a nonconstant branched conformal harmonic map from $\mathbb{S}^2$ to $\Sigma_0$, which is again impossible.  
This contradiction shows that there exists $\iota_0>0$ such that
\[
\mathrm{inj}(M_i)\ge\iota_0\qquad\text{for all }i.
\]
    
    At this point the domain sequence $(M_i,h_i)$ has uniformly bounded curvature (zero or $-1$), a uniform positive lower bound on injectivity radius  and uniformly bounded area: when $g=1$ the
metric $h_i$ is flat and normalized to have area $1$, and when $g\ge2$ the
area bound follows from Gauss-Bonnet for hyperbolic metrics of constant
curvature $-1$.  

    By Cheeger-Gromov compactness (or Mumford compactness for constant curvature surfaces), there exist a compact surface $M_0$ of genus $g$, a smooth metric $h_0$ on $M_0$, and diffeomorphisms
    \[
    \psi_i:M_0\longrightarrow M_i
    \]
    such that $\psi_i^*h_i\to h_0$ smoothly on $M_0$.
    
    Consider the maps
    \[
    \tilde f_i:=f_i\circ\psi_i:(M_0,\psi_i^*h_i)\longrightarrow \Sigma_i.
    \]

    They have uniformly bounded energy and uniformly bounded gradients (since the
$f_i$ were conformal with uniformly controlled conformal factor).  Hence, after
passing to a subsequence, $\tilde f_i$ admits a $C^0$ limit,
denoted $f_0$.

To upgrade this to strong $W^{2,2}$ convergence and to obtain uniform
convergence of the conformal factors, fix $p\in M_0$ and set
$p_i:=\psi_i(p)$.  By Cheeger-Gromov convergence, we may choose conformal
charts $\phi,\phi_i : D\to M_0,M_i$ with $\phi(0)=p$ and $\phi_i(0)=p_i$, such
that
\[
\chi_i := \phi_i^{-1}\circ\psi_i\circ\phi \longrightarrow \mathrm{Id}
\quad\text{smoothly on }D.
\]

Lemma~\ref{lem:conformal-convergence} applied in the $M_i$-charts yields
\[
f_i\circ\psi_i \to f_0\circ\phi 
   \quad\text{in }W^{2,2}(D_{1/2}),\qquad
u_i\circ\phi_i \to u_0\circ\phi 
   \quad\text{in }L^\infty(D_{1/2}),
\]
where $u_i,u_0$ denote the corresponding conformal factors.

On $D_{1/2}$ we have the identities
\[
(f_i\circ\psi_i)\circ\phi 
   = (f_i\circ\phi_i)\circ\chi_i,
\qquad
(u_i\circ\psi_i)\circ\phi 
   = (u_i\circ\phi_i)\circ\chi_i,
\]
and since $\chi_i\to\mathrm{Id}$ smoothly, composition stability of $W^{2,2}$
and $L^\infty$ implies
\[
f_i\circ\psi_i \longrightarrow f_0
   \quad\text{in } W^{2,2}\text{ near }p,
\qquad
u_i\circ\psi_i \longrightarrow u_0
   \quad\text{uniformly near }p.
\]

A finite covering of $M_0$ by such charts now yields the global convergences
\[
f_i\circ\psi_i \longrightarrow f_0 
     \quad\text{in }W^{2,2}(M_0),
\qquad
u_i\circ\psi_i \longrightarrow u_0
     \quad\text{in }L^\infty(M_0).
\]
Moreover, \(f_0\) is also the \(W^{2,2}\)-limit of \(\pi_i \circ \tilde f_i\), where \(\pi_i : \Sigma_i \to \Sigma_0\) is the nearest-point projection.  
By Proposition~\ref{prop:degree}, each \(\pi_i\) has degree \(1\); hence \(\pi_i \circ \tilde f_i\) has degree \(1\), and therefore the limit map \(f_0\) also has degree \(1\).Consequently, \(f_0\) is a global conformal diffeomorphism from \(M_0\) onto \(\Sigma_0\).
\end{proof}

\begin{theorem}[Qualitative stability of the Clifford torus]\label{Thm: qualitative-main}
Under the assumptions of Theorem~\ref{Thm: main}, 
there exists a diffeomorphism
\[
f:\mathbb{S}^1\times \mathbb{S}^1\longrightarrow\Sigma
\]
such that
\[
f^{*}g_{\mathbb{S}^3}
= e^{2u}\,(a\, d\theta^2 + 2 b\, d\theta\, d\varphi + c\, d\varphi^2),
\qquad
f_0^{*}g_{\mathbb{S}^3}=\tfrac12(d\theta^2+d\varphi^2).
\]
Moreover, as $\delta\to 0$, the following quantities tend to zero:
\begin{enumerate}
\item 
\[
\| f - f_0 \|_{W^{2,2}(\mathbb{S}^1\times \mathbb{S}^1)} \longrightarrow 0,
\]

\item 
\[
\|u\|_{L^\infty(\mathbb{S}^1\times \mathbb{S}^1)} \longrightarrow 0,
\]

\item 
\[
|a - 1/2| + |b| + |c - 1/2| \longrightarrow 0.
\]
\end{enumerate}
\end{theorem}

\begin{proof}
For $j$ large, the uniformization theorem gives conformal parametrizations
\[
f_j:(\mathbb{S}^1\times \mathbb{S}^1,h_j)\to\Sigma_j,
\]
where $h_j$ is a flat metric with $\mathrm{Area}(\mathbb{S}^1\times \mathbb{S}^1,h_j)=2\pi^2$.
Applying Theorem~\ref{thm:global-convergence} (after composing with suitable conformal
automorphisms of $\mathbb{S}^3$), we obtain a limit parametrization
\[
f_0:(\mathbb{S}^1\times \mathbb{S}^1,h_0)\to\mathbb{T}^2
\]
such that
\[
f_j\to f_0 \text{ in }W^{2,2},\qquad
u_j\to u_0 \text{ in }L^\infty,\qquad
h_j\to h_0 \text{ smoothly}.
\]

Since $f_0$ is a conformal map between flat tori of the same area, it is an isometry.  
Thus we may choose global angular coordinates $(\theta,\varphi)$ so that
\[
h_0=\tfrac12(d\theta^2+d\varphi^2),\quad u_0=0.
\]
Each $h_j$ is a flat metric with area of $2\pi^2$ on $\mathbb{S}^1\times \mathbb{S}^1$, and hence in these fixed
coordinates it has the standard constant-coefficient form
\[
h_j=a_j\,d\theta^2+2b_j\,d\theta\,d\varphi+c_j\,d\varphi^2,
\qquad a_jc_j-b_j^2=\tfrac14.
\]
The smooth convergence $h_j\to h_0$ is equivalent to
\[
a_j\to\tfrac12,\qquad b_j\to 0,\qquad c_j\to\tfrac12.
\]

Together with $f_j\to f_0$ and $u_j\to u_0$, this yields the three convergences in
Theorem~\ref{Thm: qualitative-main}, completing the proof.
\end{proof}

 \section{Almost Willmore Minimizing Implies Almost Minimal }

In this section we study the almost minimizing property for surfaces whose
Willmore energy is sufficiently close to $2\pi^{2}$.  
A central object in our analysis is the canonical five-parameter family
$\{\Sigma_{(v,t)}\}$ introduced by Marques-Neves.  
Rather than repeating the full construction and discussion from the
Introduction, we recall only the variational problem that plays a decisive role:
\begin{equation}\label{eq:max-canonical-area}
\max_{(v,t)\in \mathring{B}^4\times[-\pi,\pi]}
\mathrm{Area}\bigl(\Sigma_{(v,t)}\bigr)\,.
\end{equation}

The fundamental result of Marques-Neves, stated as Theorem~\ref{2pi2-theorem} in
the Introduction, asserts that for any smooth embedded surface of genus at least
$1$, the canonical family satisfies
\[
\max_{(v,t)} \mathrm{Area}(\Sigma_{(v,t)}) \ge 2\pi^{2},
\]
with equality if and only if the surface is  the
Clifford torus.

In this section our aim is to show that whenever 
\(\Sigma_{(v,t)}\)  attains the maximal area in
\eqref{eq:max-canonical-area}, the corresponding surface
\(\Sigma_v=\Sigma_{(v,0)}\) must be, in a quantitative sense, almost minimal.  
This constitutes the key step in passing from a nearly minimizing Willmore
surface to an almost-minimizing surface in the canonical family.

Before analyzing the maximizers of \eqref{eq:max-canonical-area}, we recall the
qualitative stability established in Section~2.  

Recall that by the qualitative stability result
(Theorem~\ref{Thm: qualitative-main}), if 
\[
\mathcal{W}(\Sigma)\le 2\pi^2+\delta_0^2,
\]
then there exists a conformal transformation 
$F:\mathbb{S}^3\to\mathbb{S}^3$ such that the conformally normalized surface
\[
\Sigma' := F(\Sigma)
\]
is close to the Clifford torus in the sense of
Theorem~\ref{Thm: qualitative-main}.  
In particular, $\Sigma'$ is $(\gamma,r_0)$-regular, where 
\[
\gamma=\gamma(\delta_0)\longrightarrow 0 
\qquad\text{as }\delta_0\to 0,
\]
while $r_0>0$ is universal.

From this point on, we work with the normalized surface $\Sigma'$ and denote
it again by $\Sigma$ for simplicity of notation.  It is this normalized
surface---already close to the Clifford torus and $(\gamma,r_0)$-regular---that
we use to construct the canonical five-parameter family 
\[
\{\Sigma_{(v,t)}\}_{(v,t)\in\mathring{B}^4\times[-\pi,\pi]}.
\]

The following proposition provides the key estimate needed to connect the
Willmore energy gap to the geometry of the canonical family.  We emphasize that
its proof does not rely on the full strength of the qualitative stability
result from Section~2: for the purpose of this proposition, it suffices that the
surface under consideration is $(\gamma,r_0)$-regular for a parameter
$\gamma=\gamma(\delta_0)$ that can be taken arbitrarily small when
$\delta_0$ is sufficiently small.

\begin{proposition}\label{away from boundary of conformal group}
There exist constants $\eta_1>0$, $\delta_0>0$, and a function 
$\gamma=\gamma(\delta_0)>0$ with $\gamma(\delta_0)\to 0$ as $\delta_0\to 0$,
and a universal radius $r_0>0$, such that the following holds.

Let $\Sigma\subset\mathbb{S}^3$ be a smooth closed torus 
and assume that $\Sigma$ is $(\gamma,r_0)$-regular and $\mathcal{W}(\Sigma)\leq 8\pi$  
Consider the canonical family 
\[
\{\Sigma_{(v,t)}\}_{(v,t)\in\mathring{B}^4\times[-\pi,\pi]},
\]
and let $(v_0,t_0)$ be any maximizer of the area , i.e.
\[
\mathrm{Area}(\Sigma_{(v_0,t_0)}) 
 = \max_{(v,t)\in\mathring{B}^4\times[-\pi,\pi]}
   \mathrm{Area}(\Sigma_{(v,t)}).
\]
Then
\[
|v_0|\le 1-\eta_1.
\]
\end{proposition}

\begin{remark}
The constant $8\pi$ is for convenience and can be replaced by any bounded constant, though $\eta_1$ would then depend on this constant.
\end{remark}

The proof of Proposition~\ref{away from boundary of conformal group} is somewhat
technical and will be presented in the Appendix.

With Proposition~\ref{away from boundary of conformal group} at hand, we can
derive a simple but important non-concentration estimate.

\begin{proposition}[Non-concentration of curvature]\label{Non-concentration of Gauss curvature}
Let $\Sigma\subset\mathbb{S}^{3}$ be the conformally normalized surface obtained
from a surface with $\mathcal{W}(\Sigma)\le 2\pi^{2}+\delta_{0}^{2}$ as in
Theorem~\ref{Thm: qualitative-main}.  In particular, $\Sigma$ satisfies the
qualitative stability conclusions of that theorem, and is therefore
$(\gamma,r_{0})$-regular with 
\[
\gamma=\gamma(\delta_{0})\longrightarrow 0 
\qquad\text{as}\quad \delta_{0}\to 0,
\]
where $r_{0}>0$ is universal.

Construct the canonical five-parameter family
\[
\{\Sigma_{(v,t)}\}_{(v,t)\in\mathring{B}^{4}\times[-\pi,\pi]},
\]
and let $(v_{0},t_{0})$ be any maximizer of the canonical area functional as in
Proposition~\ref{away from boundary of conformal group}.  
Set
\[
\Sigma_{v_{0}} := F_{v_{0}}(\Sigma),
\]
and let $A$ denote the second fundamental form of $\Sigma_{v_{0}}$.

Then there exists $\varepsilon_{0}>0$, depending only on
$\delta_{0}$ and the universal constants in Theorem~\ref{Thm: qualitative-main}
and Proposition~\ref{away from boundary of conformal group}, such that for every
measurable set $E\subset\Sigma_{v_{0}}$ with $\mathcal{H}^{2}(E)\le
\varepsilon_{0}$, one has
\[
\int_{E} |A|^{2}\, d\mathcal{H}^{2} \;\le\; \frac{1}{3}.
\]
\end{proposition}
\begin{proof}
By Theorem~\ref{Thm: qualitative-main}, for $\delta_0>0$ sufficiently small
there exists a conformal parametrization
\[
f:\mathbb{S}^1\times \mathbb{S}^1\longrightarrow\Sigma
\]
such that
\[
f^{*}g_{\mathbb{S}^{3}}
 = e^{2u}\,(a\,d\theta^{2}+2b\,d\theta\,d\varphi+c\,d\varphi^{2}),
\qquad 
f_0^{*}g_{\mathbb{S}^{3}}
 = \tfrac12(d\theta^{2}+d\varphi^{2}),
\]
with
\[
\|f-f_0\|_{W^{2,2}(T^2)}\ll 1,\qquad
\|u\|_{L^\infty(T^2)}\ll 1,\qquad
|a-\tfrac12|+|b|+|c-\tfrac12|\ll 1.
\]
Hence $f^{*}g_{\mathbb{S}^3}$ is uniformly equivalent to the flat metric on
$\mathbb{S}^1\times \mathbb{S}^1$. For any measurable $E'\subset\Sigma$ we have
\begin{equation}\label{eq:A-Sigma-via-f}
\int_{E'} |A_\Sigma|^{2}\, d\mathcal{H}^{2}
 \;\le\; C_1\int_{f^{-1}(E')} \bigl(|D^{2}(f-f_0)|^{2}+1\bigr)\, d\theta d\varphi,
\end{equation}
and
\begin{equation}\label{eq:area-compare-f}
C_1^{-1} |f^{-1}(E')| \;\le\; \mathcal{H}^{2}(E') \;\le\; C_1 |f^{-1}(E')|,
\end{equation}
for a universal $C_1\ge1$.

Since $\|f-f_0\|_{W^{2,2}}\to 0$ as $\delta_0\to 0$, there exists
$\delta_1>0$ such that for every $U\subset \mathbb{S}^1\times \mathbb{S}^1$ with $|U|\le\delta_1$,
\[
\int_{U} \bigl(|D^{2}(f-f_0)|^{2}+1\bigr)\, d\theta d\varphi
 \;\le\; \frac{1}{6C_1C_2}.
\]
for some $C_2$ to be determined. By \eqref{eq:A-Sigma-via-f}–\eqref{eq:area-compare-f}, if
$\mathcal{H}^{2}(E')\le\delta_1/C_1$ then
\[
\int_{E'} |A_\Sigma|^{2}\, d\mathcal{H}^{2} \;\le\; \frac{1}{6C_2}.
\]

We now pass from $\Sigma$ to $\Sigma_{v_0}$.  
Proposition~\ref{away from boundary of conformal group} ensures that
$|v_0|\le 1-\eta_1$ for some universal $\eta_1>0$.  
Since $F_v$ depends smoothly on $v$ and $\{|v|\le 1-\eta_1\}$ is compact, there
exists a universal $C_2$ such that for all such $v$,
\[
C_2^{-1}\,F_{v\#} d\mathcal{H}^{2}_\Sigma
 \;\le\; d\mathcal{H}^{2}_{\Sigma_v}
 \;\le\; C_2\, F_{v\#}d\mathcal{H}^{2}_\Sigma,
\qquad
|A_{\Sigma_v}(F_v(p))|
 \;\le\; C_2\bigl(|A_\Sigma(p)|+1\bigr).
\]
Applying this with $v=v_0$ (denoting $A=A_{\Sigma_{v_0}}$), let
$E\subset\Sigma_{v_0}$ be measurable and set $E' := F_{v_0}^{-1}(E)$. Then
$\mathcal{H}^{2}(E')\le C_2\,\mathcal{H}^{2}(E)$ and
\[
\int_{E} |A|^{2}\, d\mathcal{H}^{2}_{\Sigma_{v_0}}
 \;\le\; C_2 \int_{E'} \bigl(|A_\Sigma|^{2}+1\bigr)\, d\mathcal{H}^{2}.
\]

Choose
\[
\varepsilon_{0} := \frac{\delta_{1}}{C_1C_2}.
\]
If $\mathcal{H}^{2}(E)\le\varepsilon_{0}$ then
$\mathcal{H}^{2}(E')\le \delta_{1}/C_1$, and therefore
\[
\int_{E'} |A_\Sigma|^{2}\, d\mathcal{H}^{2} \le \frac{1}{6C_2},
\qquad
\int_{E'} 1\, d\mathcal{H}^{2} \le \frac{1}{6C_2}.
\]
Thus
\[
\int_{E} |A|^{2}\, d\mathcal{H}^{2}_{\Sigma_{v_0}}
 \;\le\; C_2\!\left(\frac{1}{6}+\frac{1}{6}\right)
 \;\le\; \frac{1}{3}.
\]
This proves the desired non–concentration estimate.
\end{proof}

With the non–concentration estimate in hand, we can now establish the main
result of this section: any surface whose Willmore energy is sufficiently close
to \(2\pi^{2}\) becomes, after a suitable conformal transformation, almost
area-maximizing and almost minimal in a quantitative sense.

\begin{theorem}\label{almost minimizing}
There exists $\delta_{0}>0$ such that for any smooth surface 
$\Sigma\subset\mathbb{S}^{3}$ of positive genus with 
$\mathcal{W}(\Sigma)\le 2\pi^{2}+\delta^{2}$ for some $\delta<\delta_{0}$,
there exists a conformal transformation 
$F_{v_{0}}:\mathbb{S}^{3}\to\mathbb{S}^{3}$ such that the surface
$\Sigma_{v_{0}}:=F_{v_{0}}(\Sigma)$ satisfies
\[
\int_{\Sigma_{v_{0}}} |H|^{2}\, d\mathcal{H}^{2}\;\le\; C\delta^{2},
\qquad
\mathrm{Area}(\Sigma_{v_{0}})
\;\ge\; (1-C\delta^{2})\, 2\pi^{2}.
\]
\end{theorem}

\begin{proof}[Proof of Theorem~\ref{almost minimizing}]
Let $\Sigma\subset\mathbb{S}^{3}$ be a smooth surface of positive genus with
$\mathcal{W}(\Sigma)\le 2\pi^{2}+\delta^{2}$ and $\delta<\delta_{0}$, as in the statement.  
By conformal invariance of the Willmore energy and Theorem~\ref{Thm: qualitative-main},
we may apply a conformal transformation $F:\mathbb{S}^{3}\to\mathbb{S}^{3}$ so that 
the transformed surface
\[
\Sigma' := F(\Sigma)
\]
is the conformally normalized surface provided by that theorem.  
As in the preceding discussion, we henceforth work with this normalized surface 
and continue to denote it simply by~$\Sigma$.

Fix $v\in\mathring{B}^{4}$ and recall that 
$\Sigma_{v}=F_{v}(\Sigma)$, with unit normal pointing into $A_{v}^{*}$ given by
\[
N_{v}(y)=\frac{DF_{v}(N(y))}{|DF_{v}(N(y))|}.
\]
For $t\in(-\pi,\pi)$ we consider the normal exponential map
\[
\varphi_{t}(y)=\exp_{y}(tN_{v}(y))
=\cos t\, y+\sin t\, N_{v}(y),
\qquad y\in\Sigma_{v}.
\]
Every point of $\Sigma_{(v,t)}$ lies at its closest distance $t$ from some point of
$\Sigma_{v}$ along the normal geodesic $\tau\mapsto\varphi_{\tau}(y)$.  
Thus
\[
\Sigma_{(v,t)}\subset \varphi_{t}(\Sigma_{v}^{*}(t)),
\]
where $\Sigma_{v}^{*}(t)$ consists of those points $y\in\Sigma_{v}$ that realize
the closest–point projection onto $\Sigma_{v}$.  
By \cite[Lemma~3.5]{MN14} (see also \cite{R99}), the signed Jacobian of
$\varphi_{t}$ is
\[
J_{\varphi_{t}}
=\Bigl(1+\frac{|H|^{2}}{4}\Bigr)
-\Bigl(\frac{H}{2}\cos t - \sin t\Bigr)^{2}
-\frac{\sin^{2}t}{4}(\lambda_{1}-\lambda_{2})^{2},
\]
where $\lambda_{1},\lambda_{2}$ are the principal curvatures of $\Sigma_{v}$.  
Arguing as in \cite{MN14}, for any $y\in\Sigma_{v}^{*}(t)$ one has
$J_{\varphi_{\tau}}(y)>0$ for all $\tau\in[0,t)$ and hence $J_{\varphi_{t}}(y)\ge0$.  
Therefore,
\begin{equation}\label{Area formula}
\mathcal{H}^{2}(\Sigma_{(v,t)})
\;\le\;
\int_{\Sigma_{v}^{*}(t)} J_{\varphi_{t}}\, d\mathcal{H}^{2}.
\end{equation}
We also write 
\[
E_{v}(t):=\Sigma_{v}\setminus\Sigma_{v}^{*}(t),
\]
which consists of those points not hit by a shortest normal geodesic from
$\Sigma_{(v,t)}$.

Now let $(v_{0},t_{0})$ be a maximizer of the canonical area functional as in
Proposition~\ref{away from boundary of conformal group}.  
By the $2\pi^{2}$-Theorem~\ref{2pi2-theorem} and \eqref{Area formula},
\begin{align*}
2\pi^{2}
\;\le\; \mathrm{Area}(\Sigma_{(v_{0},t_{0})})
\;\le\; \int_{\Sigma_{v_{0}}^{*}(t_{0})} J_{\varphi_{t_{0}}}\, d\mathcal{H}^{2}.
\end{align*}
Using again \cite[Lemma~3.5]{MN14} (see also \cite{R99}), we have
\[
J_{\varphi_{t_{0}}}(p)
=\Bigl(1+\frac{|H|^{2}}{4}\Bigr)
 -\Bigl(\frac{H}{2}\cos t_{0}-\sin t_{0}\Bigr)^{2}
 -\sin^{2}t_{0}\, |\mathring{A}|^{2},
\]
where $H$ is the mean curvature and $ \mathring{A}=A-\frac{H}{2}g$. Hence
\begin{align*}
2\pi^{2}
&\le \int_{\Sigma_{v_{0}}^{*}(t_{0})}
\Bigl(1+\frac{|H|^{2}}{4}\Bigr)
 -\Bigl(\frac{H}{2}\cos t_{0}-\sin t_{0}\Bigr)^{2}
 -\sin^{2}t_{0}\, |\mathring{A}|^{2}\, d\mathcal{H}^{2} \\
&\le \int_{\Sigma_{v_{0}}^{*}(t_{0})} \Bigl(1+\frac{|H|^{2}}{4}\Bigr)\, d\mathcal{H}^{2} \\
&\le \mathcal{W}(\Sigma_{v_{0}}) \\
&=\mathcal{W}(\Sigma)\;\le\; 2\pi^{2}+\delta^{2}.
\end{align*}
This yields the three estimates
\begin{align}
\int_{E_{v_{0}}(t_{0})}\Bigl(1+\frac{|H|^{2}}{4}\Bigr)\, d\mathcal{H}^{2}
&\le \delta^{2}, \label{negative small} \\
\int_{\Sigma_{v_{0}}^{*}(t_{0})}
\Bigl(\frac{H}{2}\cos t_{0}-\sin t_{0}\Bigr)^{2} d\mathcal{H}^{2}
&\le \delta^{2}, \label{mean curvature small} \\
\sin^{2}t_{0} \int_{\Sigma_{v_{0}}^{*}(t_{0})} |\mathring{A}|^{2}\, d\mathcal{H}^{2}
&\le \delta^{2}. \label{small time}
\end{align}

From \eqref{small time} and \eqref{negative small} we deduce
\begin{equation}\label{construct trace free}
\sin^{2}t_{0}
\int_{\Sigma_{v_{0}}} |\mathring{A}|^{2}\, d\mathcal{H}^{2}
\;\le\; \delta^{2}
 + \sin^{2}t_{0}\int_{E_{v_{0}}(t_{0})} |\mathring{A}|^{2}\, d\mathcal{H}^{2}.
\end{equation}
Since the trace–free part of the second fundamental form is conformally
invariant, we have
\begin{align*}
\frac{1}{2}\int_{\Sigma_{v_{0}}}|\mathring{A}|^{2}\, d\mathcal{H}^{2}
&=\frac{1}{4}\int_{\Sigma_{v_{0}}}|H|^{2}\, d\mathcal{H}^{2}
 +\mathcal{H}^{2}(\Sigma_{v_{0}})-\int_{\Sigma_{v_{0}}}K\, d\mathcal{H}^{2} \\
&=\mathcal{W}(\Sigma_{v_{0}})\;\ge\; 2\pi^{2},
\end{align*}
where we used Gauss-Bonnet (noting that $g(\Sigma_{v_{0}})=1$) and the the Willmore conjecture~\cite{MN14},
Thus, from \eqref{construct trace free},
\[
\bigl(4\pi^{2}-\int_{E_{v_{0}}(t_{0})}|A|^{2}\, d\mathcal{H}^{2}\bigr)\,
\sin^{2}t_{0}
\;\le\; \delta^{2}.
\]
By Proposition~\ref{Non-concentration of Gauss curvature} we have
\[
\int_{E_{v_{0}}(t_{0})}|A|^{2}\, d\mathcal{H}^{2}\;\le\;\frac{1}{3},
\]
and hence, for $\delta<\delta_{0}\ll1$,
\[
\sin^{2}t_{0}\;\le\; \delta^{2}
\qquad\text{and}\qquad
\cos^{2}t_{0}\;\ge\; 1-\delta^{2}.
\]

Returning to \eqref{mean curvature small}, we obtain
\begin{align*}
\frac{1-\delta^{2}}{4}
\int_{\Sigma_{v_{0}}^{*}(t_{0})}|H|^{2}\, d\mathcal{H}^{2}
&\le \delta^{2}
 + \int_{\Sigma_{v_{0}}^{*}(t_{0})} H\cos t_{0}\sin t_{0}\, d\mathcal{H}^{2} \\
&\le \delta^{2}
 + \frac{1}{16}\int_{\Sigma_{v_{0}}^{*}(t_{0})}|H|^{2}\, d\mathcal{H}^{2}
 + 4\sin^{2}t_{0}\,\mathcal{H}^{2}(\Sigma_{v_{0}}),
\end{align*}
and therefore, for $\delta<\delta_{0}\ll1$,
\[
\frac{1}{8}\int_{\Sigma_{v_{0}}^{*}(t_{0})}|H|^{2}\, d\mathcal{H}^{2}
\;\le\; \delta^{2}
 + 4\delta^{2}\,\mathcal{W}(\Sigma_{v_{0}}).
\]
Combining this with \eqref{negative small} yields
\[
\int_{\Sigma_{v_{0}}}|H|^{2}\, d\mathcal{H}^{2}
\;\le\; C\,\delta^{2}
\]
for some universal constant $C$.

Finally, by the Willmore conjecture~\cite{MN14},
\[
\mathcal{W}(\Sigma_{v_{0}})=\mathcal{W}(\Sigma)\ge 2\pi^{2},
\]
so
\[
\mathrm{Area}(\Sigma_{v_{0}})
= \mathcal{W}(\Sigma_{v_{0}})
 - \frac{1}{4}\int_{\Sigma_{v_{0}}}|H|^{2}\, d\mathcal{H}^{2}
\;\ge\; 2\pi^{2}-C\delta^{2}.
\]
This proves the theorem.
\end{proof}

\section{From Almost-Minimizing to Quantitative Stability}
In this section we prove the quantitative stability result stated in
Theorem~\ref{Thm: main} by performing a linearized analysis in a neighbourhood
of the Clifford torus.
Denote the standard immersion $f_0:\mathbb{S}^1\times \mathbb{S}^1\to \mathbb{S}^3$ by 
$$f_0(\theta,\varphi)=\frac{1}{\sqrt{2}}(\cos \theta, \sin\theta, \cos \varphi, \sin \varphi)=:(\phi, \psi),$$
where $\phi=\frac{1}{\sqrt{2}}(\cos \theta, \sin\theta)$ and $\psi=\frac{1}{\sqrt{2}}(\cos \varphi, \sin \varphi)$ satisfies $|\phi|^2=|\psi|^2=\frac{1}{2}$. 
Then, 
\begin{align*}
f_{0\theta}&=(\psi_\theta,0)=\frac{1}{\sqrt{2}}(-\sin \theta, \cos \theta, 0, 0)\\
f_{0\varphi}&=(0, \psi_\varphi)=\frac{1}{\sqrt{2}}(0,0,-\sin \varphi, \cos \varphi).
\end{align*}
Thus by $\phi_\theta \cdot\phi=\psi_{\varphi}\cdot \psi=0$, we know $$n=(-\phi, \psi)$$
satisfies $n\cdot f_0=n\cdot f_{0\theta}=n\cdot f_{0\varphi}=0$, i.e., $n\in T_{f_0}\mathbb{S}^{3}\cap T_{f_0}^\bot \mathbb{T}^2$ is a unit normal vector of  $\mathbb{T}^2\subset \mathbb{S}^3$. 

Now, for any immersion $f:\mathbb{S}^1\times \mathbb{S}^1\to \mathbb{S}^3$, $f^{*}g$ is conformal to some flat torus $(\mathbb{S}^1\times \mathbb{S}^1,g_\tau)$, where $$g_\tau=ad\theta\otimes d\theta+b(d\theta\otimes d\varphi+d\varphi\otimes d\theta)+cd\varphi\otimes d\varphi$$ with  $ac-b^2=\frac{1}{4}$. In precisely, there exists a function $u$ on $\mathbb{S}^1\times \mathbb{S}^1$ such that 
$$df\otimes df=e^{2u}g_\tau.$$
\begin{proposition}\label{quantitative complex structure and area}
Assume $f = f_0 + h : \mathbb{S}^1\times \mathbb{S}^1 \to \mathbb{S}^3$ is an immersion with 
$\|h\|_{W^{1,2}(\mathbb{S}^1\times \mathbb{S}^1)} + \|H\|_{L^1(d\mu_f)} \ll 1$. Then 
$$|b| + |a - \tfrac{1}{2}| + |c - \tfrac{1}{2}| + |\mathrm{Area}(\Sigma) - 2\pi^2| \le C(\|H\|_{L^1(d\mu_f)} + \|h\|^2_{W^{1,2}(\mathbb{S}^1\times \mathbb{S}^1)}),$$
where $\Sigma = f(\mathbb{S}^1\times \mathbb{S}^1) \subset \mathbb{S}^3$.
\end{proposition}

\begin{proof}
Decompose $h = v + z n + w f_0$ with $w = -\tfrac{1}{2}|h|^2$ and $v = v^1 f_{0\theta} + v^2 f_{0\varphi}$ tangent. Using the derivatives of $f_0$ and $n$:
\begin{align*}
f_{0\theta} &= (-\phi_\theta, 0), \quad f_{0\varphi} = (0, \psi_\varphi), \quad |f_{0\theta}|^2 = |f_{0\varphi}|^2 = \tfrac{1}{2}, \\
n_\theta &= -f_{0\theta}, \quad n_\varphi = f_{0\varphi}, \\
f_{0\theta\theta} &= -\tfrac{1}{2}(f_0 - n), \quad f_{0\varphi\varphi} = -\tfrac{1}{2}(f_0 + n), \quad f_{0\theta\varphi} = 0.
\end{align*}
We compute the derivatives of $h$:
\begin{align*}
h_\theta &= (v^1_\theta - z + w) f_{0\theta} + v^2_\theta f_{0\varphi} + (w_\theta - \tfrac{v^1}{2}) f_0 + (z_\theta + \tfrac{v^1}{2}) n, \\
h_\varphi &= v^1_\varphi f_{0\theta} + (v^2_\varphi + z + w) f_{0\varphi} + (w_\varphi - \tfrac{v^2}{2}) f_0 + (z_\varphi - \tfrac{v^2}{2}) n.
\end{align*}
The conformal relations $df \otimes df = e^{2u} g_\tau$ yield:
\begin{align}
e^{2u} a - \tfrac{1}{2} &= 2 f_{0\theta} \cdot h_\theta + |h_\theta|^2 = v^1_\theta - z + w + |h_\theta|^2, \label{eq:a} \\
e^{2u} c - \tfrac{1}{2} &= 2 f_{0\varphi} \cdot h_\varphi + |h_\varphi|^2 = v^2_\varphi + z + w + |h_\varphi|^2, \label{eq:c} \\
e^{2u} b &= f_{0\theta} \cdot h_\varphi + f_{0\varphi} \cdot h_\theta + h_\theta \cdot h_\varphi = \tfrac{1}{2}(v^1_\varphi + v^2_\theta) + h_\theta \cdot h_\varphi. \label{eq:b}
\end{align}
Integrating over $\mathbb{S}^1\times \mathbb{S}^1$ and using periodicity (so integrals of $v^1_\theta$, $v^2_\varphi$, etc. vanish):
\begin{align*}
(2a - 1)A + \int z &= \int (|h_\theta|^2 - \tfrac{1}{2}|h|^2) - (A - 2\pi^2), \\
(2c - 1)A - \int z &= \int (|h_\varphi|^2 - \tfrac{1}{2}|h|^2) - (A - 2\pi^2), \\
2b A &= \int h_\theta \cdot h_\varphi,
\end{align*}
where $A = \mathrm{Area}(\Sigma) = \tfrac{1}{2} \int e^{2u}$. Immediately:
\begin{equation}\label{eq:b-est}
|b| \le C \|dh\|_{L^2}^2.
\end{equation}

Adding \eqref{eq:a} and \eqref{eq:c}:
\begin{equation}\label{eq:sum}
e^{2u}(a + c) - 1 = v^1_\theta + v^2_\varphi + 2w + |dh|^2 = O(|dh| + |dh|^2).
\end{equation}
Integration yields (noting $\int e^{2u} = 2A$ and $\int (v^1_\theta + v^2_\varphi) = 0$):
\begin{equation}\label{eq:a+c}
|(a + c)A - 2\pi^2| \le C \|h\|_{W^{1,2}}^2, \quad \text{so} \quad |a + c - 1| \le C(\|h\|_{W^{1,2}} + |A - 2\pi^2|).
\end{equation}

Now use the mean curvature $H = \Delta_g f + 2f$. The first variation gives:
\begin{align*}
I &:= \int H \cdot n \, d\mu_g = \int [-\nabla_g f \cdot \nabla_g n + 2z] \, d\mu_g=\int[-\nabla_{g_\tau} f\cdot \nabla_{g_\tau} n+2ze^{2u}]d\mu_{g_\tau}.
\end{align*}
Compute in coordinates (using $d\mu_{g_\tau} = \tfrac{1}{2}d\theta d\varphi$ and the metric $g_\tau$ with components $a,b,c$), we have (using the expressions for $h$ and the derivatives of $n$):
\begin{align*}
I &= \int \left[ (c - a)(1 - \tfrac{1}{2}|h|^2) + (e^{2u} - (a + c))z \right] d\theta d\varphi + O(\|h\|_{W^{1,2}}^2).
\end{align*}
From \eqref{eq:sum} and \eqref{eq:a+c}, we have $e^{2u} - (a + c) = s + O(|dh|)$ with $s = \frac{1 - (a + c)^2}{a + c}$. Thus by noting $|z|\le |h|$, we have
\begin{equation}\label{eq:I}
I = (c - a)(4\pi^2 - \tfrac{1}{2}\|h\|_{L^2}^2) + 4\pi^2 s \bar{z} + O(\|h\|_{W^{1,2}}^2),
\end{equation}
where $\bar{z} = \frac{1}{4\pi^2} \int z$. Since $|I| \le \|H\|_{L^1}$, we get:
\begin{equation}\label{eq:c-a-rough}
|c - a| \le C(\|H\|_{L^1} + |s||\bar{z}| + \|h\|_{W^{1,2}}^2).
\end{equation}

Subtracting the integrated conformal relations:
\begin{equation}\label{eq:c-a+z}
(c - a)A + \int z = \tfrac{1}{2} \int (|h_\theta|^2 - |h_\varphi|^2) \quad \Rightarrow \quad |(c - a)A + \int z| \le C \|dh\|_{L^2}^2.
\end{equation}
Combining \eqref{eq:c-a-rough} and \eqref{eq:c-a+z}, and noting that
$|s| \le C|a + c - 1| \le C(\|h\|_{W^{1,2}} + |A - 2\pi^2|)$.  Since  $\|h\|_{W^{1,2}}$ is small, we may invoke the dominated convergence theorem, which yields that the area $A =\frac{1}{2}\int e^{2u}= \int|df|^2$ is close to $2\pi^2$, in particular, $|A - 2\pi^2|$ is small.
Consequently, the quantity $s = \frac{1 - (a + c)^2}{a + c}$ is small in a qualitative sense. Thus
\begin{equation}\label{eq:c-a-z}
|c - a| + |\bar{z}| \le C(\|H\|_{L^1} + \|h\|_{W^{1,2}}^2).
\end{equation}
From $ac - b^2 = \tfrac{1}{4}$ and \eqref{eq:b-est}, we obtain:
\begin{equation}\label{eq:a-c-precise}
\max\{|a - \tfrac{1}{2}|, |c - \tfrac{1}{2}|\} \le |c - a| + b^2 \le C(\|H\|_{L^1} + \|h\|_{W^{1,2}}^2).
\end{equation}

Finally, from \eqref{eq:sum} and the area formula $A = \tfrac{1}{2} \int e^{2u}$:
\begin{equation}\label{eq:area}
|A - 2\pi^2| \le C(\|H\|_{L^1} + \|h\|_{W^{1,2}}^2).
\end{equation}

The proposition follows from \eqref{eq:b-est}, \eqref{eq:a-c-precise}, and \eqref{eq:area}.
\end{proof}

Since $h = v + z n + w f_0$ with $w = -\frac{|h|^2}{2}$ being higher order, 
we estimate $v$ and $z$. Rewriting the conformal relations gives the Cauchy-Riemann system for $v$:

\begin{align}\label{Cauchy Riemann equation of v}
\begin{aligned}
v^1_\theta - v^2_\varphi &= e^{2u}(a-c) + 2z - |h_\theta|^2 + |h_\varphi|^2, \\
v^1_\varphi + v^2_\theta &= 2e^{2u}b - 2h_\theta \cdot h_\varphi.
\end{aligned}
\end{align}

Since $|a-c|, |b|$ modulo higher order terms controlled by $\|H\|_{L^1}$, 
$v$ is well-estimated once $z$ is controlled, with the estimation error being of higher order.
The key is to derive an equation for $z$ via normal variation.

\begin{lemma}\label{z equation}
Let $f = f_0 + h: \mathbb{S}^1\times \mathbb{S}^1 \to \mathbb{S}^3$ be an immersion with 
$g_f = df \otimes df = e^{2u} g_\tau$ where 
$g_\tau = a d\theta \otimes d\theta + b(d\theta \otimes d\varphi + d\varphi \otimes d\theta) + c d\varphi \otimes d\varphi$, 
$ac - b^2 = \frac{1}{4}$. Assume:
\begin{align}\label{qualitative w22 conformal small}
\|h\|_{W^{2,2}} + \|H\|_{L^1(d\mu_f)} \ll 1 \quad \text{and} \quad \|u\|_{L^\infty} \le C.
\end{align}
Then for any $k \in C^\infty(\mathbb{S}^1\times \mathbb{S}^1)$:
\[
\int_{\mathbb{S}^1\times \mathbb{S}^1} (\nabla k \nabla z - (2+\psi)kz) \, d\theta d\varphi 
= O(\|k\|_{W^{1,2}} \cdot (\|H\|_{L^2(\mu_g)} + \|h\|_{W^{2,2}}^2)),
\]
where $z = h \cdot n$ and $\psi = e^{2u} - 1$.
\end{lemma}

\begin{proof}
Using $H = \Delta_{g_f} f + 2f$ and $\|u\|_{L^\infty} \le C$:
\[
|I| := \left|\int (\Delta_{g_f} f + 2f) \cdot kn \, d\mu_{g_f}\right| 
\le \|H\|_{L^2(\mu_{g_f})} \|k\|_{L^2(\mu_{g_f})} 
\le C \|H\|_{L^2(\mu_{g_f})} \|k\|_{L^2(\mathbb{S}^1\times \mathbb{S}^1)}.
\]

Compute $I$ using conformal invariance ($d\mu_{g_\tau} = \frac{1}{2} d\theta d\varphi$):
\begin{align*}
I &= \int \Delta_{g_\tau} f \cdot kn \, d\mu_{g_\tau} + 2\int k h \cdot n e^{2u} \, d\mu_{g_\tau} \\
&= \underbrace{\int \Delta_{g_\tau}(f_0 + h) \cdot kn \, d\mu_{g_\tau}}_{\text{II}} + \int k z e^{2u} \, d\theta d\varphi.
\end{align*}

Let $g_0 = \frac{1}{2}(d\theta^2 + d\varphi^2)$, then $\Delta_{g_0} f_0 = -2f_0$ and  $\Delta_{g_0} f_0 \cdot kn = 0$. Thus:
\begin{align*}
\text{II} &= \int (\Delta_{g_\tau} - \Delta_{g_0}) f_0 \cdot kn \, d\mu_{g_\tau} - \int \nabla_{g_\tau} h \cdot \nabla_{g_\tau}(kn) \, d\mu_{g_\tau} \\
&= \underbrace{\int (g_\tau^{ij} - 2\delta^{ij}) f_{0,ij} \cdot kn \, d\mu_{g_\tau}}_{\text{III}} 
- \underbrace{\int g_\tau^{ij} h_i \cdot (kn)_j \, d\mu_{g_\tau}}_{\text{IV}}.
\end{align*}

Using $f_{0,\theta\theta} = -\frac{f_0 - n}{2}$, $f_{0,\varphi\varphi} = -\frac{f_0 + n}{2}$, 
$f_{0,\theta\varphi} = 0$, and $g_\tau^{\theta\theta} = 4c$, $g_\tau^{\varphi\varphi} = 4a$, $g_\tau^{\theta\varphi} = -4b$:
\begin{align*}
\text{III} &= \int [(4c-2)f_{0,\theta\theta} \cdot kn + (4a-2)f_{0,\varphi\varphi} \cdot kn] \, d\mu_{g_\tau} \\
&= \int \left[(4c-2)\cdot \frac{k}{2} + (4a-2)\cdot \left(-\frac{k}{2}\right)\right] \frac{1}{2} d\theta d\varphi \\
&= (c - a) \int k \, d\theta d\varphi.
\end{align*}
By Proposition \ref{quantitative complex structure and area}:
\[
|\text{III}| = O(\|k\|_{L^1} \cdot (\|H\|_{L^1(\mu_g)} + \|h\|_{W^{1,2}}^2)).
\]

Split IV:
\begin{align*}
\text{IV} &= \int (g_\tau^{ij} - 2\delta^{ij}) h_i \cdot (kn)_j \, d\mu_{g_\tau} 
+ \int [h_\theta \cdot (kn)_\theta + h_\varphi \cdot (kn)_\varphi] \, d\theta d\varphi \\
&= O(\|k\|_{W^{1,2}} \cdot \|dh\|_{L^2} \cdot (\|H\|_{L^1} + \|h\|_{W^{1,2}}^2)) + \text{V},
\end{align*}
where 
\[
\text{V} = \int [h_\theta \cdot (kn)_\theta + h_\varphi \cdot (kn)_\varphi] \, d\theta d\varphi.
\]

Using the expressions for $h_\theta$, $h_\varphi$ and $(kn)_\theta = k_\theta n - k f_{0\theta}$, $(kn)_\varphi = k_\varphi n + k f_{0\varphi}$:
\begin{align*}
h_\theta \cdot (kn)_\theta &= k_\theta z_\theta + \frac{k(z-w)}{2} + \frac{(kv^1)_\theta}{2} - k v^1_\theta, \\
h_\varphi \cdot (kn)_\varphi &= k_\varphi z_\varphi + \frac{k(z+w)}{2} - \frac{(kv^2)_\varphi}{2} + k v^2_\varphi.
\end{align*}

Summing and using \eqref{Cauchy Riemann equation of v}:
\begin{align*}
&h_\theta \cdot (kn)_\theta + h_\varphi \cdot (kn)_\varphi \\
&= \nabla k \nabla z + kz + \frac{(kv^1)_\theta - (kv^2)_\varphi}{2} + k(v^2_\varphi - v^1_\theta) \\
&= \nabla k \nabla z + kz + \frac{(kv^1)_\theta - (kv^2)_\varphi}{2} - k(e^{2u}(a-c) + 2z - |h_\theta|^2 + |h_\varphi|^2) \\
&= \nabla k \nabla z - kz + \frac{(kv^1)_\theta - (kv^2)_\varphi}{2} - k(e^{2u}(a-c) - |h_\theta|^2 + |h_\varphi|^2).
\end{align*}

Thus:
\begin{align*}
\text{V} &= \int (\nabla k \nabla z - kz) \, d\theta d\varphi + (c-a) \int k e^{2u} \, d\theta d\varphi \\
&\quad + O(\|k\|_{L^2} \cdot \|dh\|_{L^4}^2).
\end{align*}

Combining III and IV:
\begin{align*}
\text{II} &= \text{III} - \text{IV} \\
&= (c-a) \int k \, d\theta d\varphi - \text{V} + O(\|k\|_{W^{1,2}} \cdot (\|H\|_{L^1} + \|h\|_{W^{1,2}}^2)) \\
&= -\int (\nabla k \nabla z - kz) \, d\theta d\varphi - (c-a) \int k(e^{2u} - 1) \, d\theta d\varphi \\
&\quad + O(\|k\|_{W^{1,2}} \cdot (\|H\|_{L^1} + \|h\|_{W^{2,2}}^2)).
\end{align*}

Finally, since $I = \text{II} + \int k z e^{2u} \, d\theta d\varphi$:
\begin{align*}
\int (\nabla k \nabla z - (2+\psi)kz) \, d\theta d\varphi 
= O(\|k\|_{W^{1,2}} \cdot (\|H\|_{L^2(\mu_g)} + \|h\|_{W^{2,2}}^2)).
\end{align*}
\end{proof}

From Lemma \ref{z equation}, we obtain:
\begin{align*}
\|\Delta z + (2+\psi)z\|_{W^{-1,2}} \le C(\|H\|_{L^2(\mu_g)} + \|h\|_{W^{2,2}}^2).
\end{align*}

To handle the kernel of $\Delta+2$ on $(\mathbb{S}^1\times \mathbb{S}^1, d\theta^2 + d\varphi^2)$, consider:
\begin{align*}
I(R) = \int_{\mathbb{S}^1\times \mathbb{S}^1} |Rf - f_0|^2 \, d\theta d\varphi, \quad R \in SO(4).
\end{align*}
Since $I(R)$ is continuous on the compact manifold $SO(4)$, there exists $R_0 \in SO(4)$ minimizing $I(R)$:
\begin{align*}
I(R_0) \le I(R) \quad \forall R \in SO(4).
\end{align*}
In particular, $I(R_0) \le \int |f - f_0|^2$, so:
\begin{align*}
\|(R_0 - id)f_0\|_{L^2} &\le \|R_0(f - f_0)\|_{L^2} + \|R_0f - f_0\|_{L^2} \\
&\le 2\|f - f_0\|_{L^2} \ll 1.
\end{align*}
Since $\Delta((R_0 - id)f_0) = -2(R_0 - id)f_0$, we have:
\begin{align*}
\|(R_0 - id)f_0\|_{W^{2,2}} \le C\|f - f_0\|_{L^2} \ll 1.
\end{align*}
Therefore:
\begin{align*}
\|R_0f - f_0\|_{W^{2,2}} &\le \|R_0(f - f_0)\|_{W^{2,2}} + \|(R_0 - id)f_0\|_{W^{2,2}} \\
&\le C\|f - f_0\|_{W^{2,2}},
\end{align*}
and $d(R_0f) \otimes d(R_0f) = df \otimes df = e^{2u}g_\tau$. Replacing $f$ by $R_0f$, we may assume:
\begin{equation}\label{Modulo rotation}
\int_{\mathbb{S}^1\times \mathbb{S}^1} |f - f_0|^2 \le \int_{\mathbb{S}^1\times \mathbb{S}^1} |Rf - f_0|^2 \quad \forall R \in SO(4).
\end{equation}

By Theorem~\ref{almost minimizing}, after a suitable conformal normalization
the surface $\Sigma$ satisfies
\[
\int_{\Sigma}|H|^{2}\le C\delta^{2},
\qquad 
\mathrm{Area}(\Sigma)\ge 2\pi^{2}-C\delta^{2},
\]
so its mean curvature is small and, by the qualitative stability results of
Section~2, the surface is already known to lie close to the Clifford torus.
Thus the remaining task is to upgrade this qualitative closeness, together
with the smallness of the mean curvature, to a fully quantitative estimate.
The next theorem provides precisely this, completing the proof of
Theorem~\ref{Thm: main}.

\begin{theorem}\label{W22 quantitative}
Let $f: \mathbb{S}^1\times \mathbb{S}^1 \to \mathbb{S}^3$ be an immersion with $g_f = e^{2u}g_\tau$ where 
$g_\tau = a d\theta^2 + 2b d\theta d\varphi + c d\varphi^2$, $ac - b^2 = \frac{1}{4}$.
Assume \eqref{Modulo rotation} holds and:
\begin{align*}
\|f - f_0\|_{W^{2,2}} + \|u\|_{L^\infty} + \|H\|_{L^1(d\mu_f)} \ll 1.
\end{align*}
Then:
\begin{align*}
\|f - f_0\|_{W^{2,2}} + |b| + |a - \tfrac{1}{2}| + |c - \tfrac{1}{2}| + |\mathrm{Area}(\Sigma) - 2\pi^2| \le C\|H\|_{L^2(d\mu_f)}.
\end{align*}
\end{theorem}

\begin{proof}
Differentiating \eqref{Modulo rotation} at $t = 0$ for $A \in \mathfrak{so}(4)$:
\begin{align*}
0 = \frac{d}{dt}\bigg|_{t=0} \int |e^{tA}f - f_0|^2 = -2\int Af \cdot f_0.
\end{align*}
Since $f = f_0 + h = v + zn + (1 + w)f_0$ and $Af_0 \cdot f_0 = 0$:
\begin{equation}\label{balance condition}
\int (v \cdot Af_0 + z n \cdot Af_0) \, d\theta d\varphi = -\int A(v + zn) \cdot f_0 = 0.
\end{equation}

Let $\{\tau_1, \tau_2, \tau_3, \tau_4\}$ be the standard basis of $\mathbb{R}^4$. For $1 \le i < j \le 4$, define $A_{ij} \in \mathfrak{so}(4)$ by:
\begin{align*}
A_{ij}(\tau_i) = \tau_j, \quad A_{ij}(\tau_j) = -\tau_i, \quad A_{ij}(\tau_k) = 0 \ (k \ne i,j).
\end{align*}
Compute:
\begin{align*}
A_{12}f_0 &= \sqrt{2}f_{0\theta}, \quad &A_{34}f_0 &= \sqrt{2}f_{0\varphi}, \\
A_{13}f_0 &= (-\cos\varphi, 0, \cos\theta, 0), \quad &A_{14}f_0 &= (-\sin\varphi, 0, 0, \cos\theta), \\
A_{23}f_0 &= (0, -\cos\varphi, \sin\theta, 0), \quad &A_{24}f_0 &= (0, -\sin\varphi, 0, \sin\theta).
\end{align*}

Write $v + zn = v^1f_{0\theta} + v^2f_{0\varphi} + zn$ in coordinates:
\begin{align*}
v + zn = \frac{1}{\sqrt{2}}(-z\cos\theta &- v^1\sin\theta,\ -z\sin\theta + v^1\cos\theta, \\
&z\cos\varphi - v^2\sin\varphi,\ z\sin\varphi + v^2\cos\varphi).
\end{align*}
Then:
\begin{align*}
(v + zn) \cdot A_{12}f_0 &= \tfrac{\sqrt{2}}{2}v^1, \\
(v + zn) \cdot A_{34}f_0 &= \tfrac{\sqrt{2}}{2}v^2, \\
(v + zn) \cdot A_{13}f_0 &= \tfrac{1}{\sqrt{2}}(2z\cos\theta\cos\varphi + v^1\sin\theta\cos\varphi - v^2\cos\theta\sin\varphi), \\
(v + zn) \cdot A_{14}f_0 &= \tfrac{1}{\sqrt{2}}(2z\cos\theta\sin\varphi + v^1\sin\theta\sin\varphi + v^2\cos\theta\cos\varphi), \\
(v + zn) \cdot A_{23}f_0 &= \tfrac{1}{\sqrt{2}}(2z\sin\theta\cos\varphi - v^1\cos\theta\cos\varphi - v^2\sin\theta\sin\varphi), \\
(v + zn) \cdot A_{24}f_0 &= \tfrac{1}{\sqrt{2}}(2z\sin\theta\sin\varphi - v^1\cos\theta\sin\varphi + v^2\sin\theta\cos\varphi).
\end{align*}

From \eqref{balance condition} with $A_{12}$ and $A_{34}$:
\begin{equation}\label{intrinsic rotation balance}
\int v^1 \, d\theta d\varphi = \int v^2 \, d\theta d\varphi = 0.
\end{equation}
For $A_{13}, A_{14}, A_{23}, A_{24}$, we obtain:
\begin{align*}
2\int z\cos\theta\cos\varphi &= \int (v^2_\varphi - v^1_\theta)\cos\theta\cos\varphi, \\
2\int z\cos\theta\sin\varphi &= \int (v^2_\varphi - v^1_\theta)\cos\theta\sin\varphi, \\
2\int z\sin\theta\cos\varphi &= \int (v^2_\varphi - v^1_\theta)\sin\theta\cos\varphi, \\
2\int z\sin\theta\sin\varphi &= \int (v^2_\varphi - v^1_\theta)\sin\theta\sin\varphi.
\end{align*}

Substitute from \eqref{Cauchy Riemann equation of v}:
\begin{align*}
v^2_\varphi - v^1_\theta = -e^{2u}(a-c) - 2z + |h_\theta|^2 - |h_\varphi|^2.
\end{align*}
Thus:
\begin{align*}
4\int z\cos\theta\cos\varphi &= \int [|h_\theta|^2 - |h_\varphi|^2 - (a-c)e^{2u}]\cos\theta\cos\varphi \\
&= O(\|h\|_{W^{1,2}}^2 + \|H\|_{L^1}), \\
4\int z\cos\theta\sin\varphi &= O(\|h\|_{W^{1,2}}^2 + \|H\|_{L^1}), \\
4\int z\sin\theta\cos\varphi &= O(\|h\|_{W^{1,2}}^2 + \|H\|_{L^1}), \\
4\int z\sin\theta\sin\varphi &= O(\|h\|_{W^{1,2}}^2 + \|H\|_{L^1}).
\end{align*}

The spectrum of $-\Delta$ on $(\mathbb{S}^1\times \mathbb{S}^1, d\theta^2 + d\varphi^2)$ is $\{m^2 + n^2 : m,n \in \mathbb{Z}_{\ge 0}\}$, and:
\[
\mathcal{N} := \mathrm{Ker}(\Delta + 2) = \mathrm{span}\{\cos\theta\cos\varphi, \cos\theta\sin\varphi, \sin\theta\cos\varphi, \sin\theta\sin\varphi\}.
\]
Let $P_{\mathcal{N}}: L^2(\mathbb{S}^1\times \mathbb{S}^1) \to \mathcal{N}$ be the orthogonal projection. Then:
\begin{align*}
\|P_{\mathcal{N}}(z)\|_{L^2} \le C(\|h\|_{W^{1,2}}^2 + \|H\|_{L^1}).
\end{align*}

For $z^\bot = z - P_{\mathcal{N}}(z)$, we have $(\Delta + 2)z^\bot = (\Delta + 2)z$. By Lemma \ref{z equation}:
\begin{align*}
\|(\Delta + 2)z^\bot\|_{W^{-1,2}} \le C(\|h\|_{W^{2,2}}^2 + \|H\|_{L^2} + \|\psi\|_{L^\infty} \|z\|_{L^2}).
\end{align*}
Since $\Delta + 2: W^{1,2} \cap \mathcal{N}^\bot \to W^{-1,2}$ is invertible:
\begin{align*}
\|z^\bot\|_{W^{1,2}} \le C(\|h\|_{W^{2,2}}^2 + \|H\|_{L^2} + \|\psi\|_{L^\infty} \|z\|_{L^2}).
\end{align*}
As $\dim\mathcal{N} = 4 < \infty$:
\begin{align*}
\|P_{\mathcal{N}}(z)\|_{W^{1,2}} \le C\|P_{\mathcal{N}}(z)\|_{L^2} \le C(\|H\|_{L^1} + \|h\|_{W^{1,2}}^2).
\end{align*}
Therefore:
\begin{align*}
\|z\|_{W^{1,2}} &\le \|P_{\mathcal{N}}(z)\|_{W^{1,2}} + \|z^\bot\|_{W^{1,2}} \\
&\le C(\|h\|_{W^{2,2}}^2 + \|H\|_{L^2} + \|\psi\|_{L^\infty} \|z\|_{L^2}).
\end{align*}
Since $\|u\|_{L^\infty} \ll 1$ implies $\|\psi\|_{L^\infty} \ll 1$:
\begin{align*}
\|z\|_{W^{1,2}} \le C(\|h\|_{W^{2,2}}^2 + \|H\|_{L^2}). 
\end{align*}

From \eqref{Cauchy Riemann equation of v}, \eqref{intrinsic rotation balance}, and Proposition \ref{quantitative complex structure and area}:
\begin{align*}
\|v\|_{W^{1,2}} &\le C(\|z\|_{L^2} + \|dh\|_{L^4}^2 + |a-c| + |b|) \\
&\le C(\|H\|_{L^2} + \|h\|_{W^{2,2}}^2). \tag{4}
\end{align*}

Since $w = -\tfrac{1}{2}|h|^2$ is higher order, we obtain
\begin{align}\label{W12 estimate with W22 square error}
\|h\|_{W^{1,2}} \le \|v\|_{W^{1,2}} + \|z\|_{W^{1,2}} + \|w\|_{W^{1,2}} \le C(\|H\|_{L^2} + \|h\|_{W^{2,2}}^2). 
\end{align}

The mean curvature equation gives
\begin{align*}
H e^{2u} = \Delta_{g_\tau}(f_0 + h) + 2(f_0 + h) = (\Delta_{g_\tau} - \Delta_{g_0})(f_0 + h) + \Delta_{g_0} h + 2h.
\end{align*}
Since $\|u\|_{L^\infty} \le C$:
\begin{align*}
\|h\|_{W^{2,2}} &\le C(\|h\|_{L^2} + \|\Delta h\|_{L^2}) \\
&\le C(\|h\|_{L^2} + \|H\|_{L^2} + (|a-\tfrac{1}{2}| + |c-\tfrac{1}{2}| + |b|)(1 + \|h\|_{W^{2,2}})).
\end{align*}
By Proposition \ref{quantitative complex structure and area}:
\begin{align*}
\|h\|_{W^{2,2}} \le C(\|h\|_{L^2} + \|H\|_{L^2} + \|H\|_{L^1} + \|dh\|_{L^2}^2),
\end{align*}
so:
\begin{align}\label{W22 dominated by w12}
\|h\|_{W^{2,2}} \le C(\|h\|_{W^{1,2}} + \|H\|_{L^2}). 
\end{align}

Substitute \eqref{W12 estimate with W22 square error} into \eqref{W22 dominated by w12}:
\begin{align*}
\|h\|_{W^{2,2}} \le C(\|h\|_{W^{2,2}}^2 + \|H\|_{L^2}),
\end{align*}
which implies (since $\|h\|_{W^{2,2}} \ll 1$):
\begin{align*}
\|h\|_{W^{2,2}} \le C\|H\|_{L^2}.
\end{align*}

Finally, by Proposition \ref{quantitative complex structure and area}:
\begin{align*}
|b| + |a - \tfrac{1}{2}| + |c - \tfrac{1}{2}| + |\mathrm{Area}(\Sigma) - 2\pi^2| \le C\|H\|_{L^2(d\mu_f)}.
\end{align*}

We now estimate the conformal factor \(u\).  
Let \(f_1:=f_\theta,f_2:=f_\varphi\) denote the two coordinate derivatives of \(f\), and set
\[
e_i=\sqrt{2}f_ie^{-u}, \qquad i=1,2.
\]
Since the metric \(g_f\) has constant coefficients in this coordinate basis, an orthonormal frame \(\{E_1,E_2\}\) for \(g_f\) can be chosen as a constant linear combination of \(\{e_1,e_2\}\).  
As any vector-valued function satisfies \(dv\wedge dv=0\), this implies
\[
\ast(dE_1\wedge dE_2)=\frac{1}{2}\ast(de_1\wedge de_2)/(ac-b^2)^{\frac{1}{2}}=\ast (de_1\wedge de_2).
\]

Because \(g_f=e^{2u}g_\tau\), the Gauss equation gives
\[
-\Delta_{g_\tau}u=\ast(de_1\wedge de_2).
\]

For the reference map \(f_0\), define similarly
\[
e_i^0=\frac{(f_0)_i}{|(f_0)_i|}=\sqrt{2}(f_0)_i,
\qquad\text{and note that }de_1^0\wedge de_2^0=0.
\]
Subtracting the corresponding identities for \(f\) and \(f_0\), we obtain
\[
-\Delta_{g_\tau}u
=\ast\bigl(d(e_1-e_1^0)\wedge de_2\bigr)
  +\ast\bigl(de_1^0\wedge d(e_2-e_2^0)\bigr).
\]

Since \(|f_i|= e^u\) and \(|u|\ll1\), we have the uniform lower bound \(|f_i|\ge c>0\).  
Hence,
\[
\begin{aligned}
|d(e_i-e_i^0)|
&=\sqrt{2}\left|\nabla\!\left(f_ie^{-u}-(f_0)_i\right)\right| \\
&\le C|\nabla(f_i-(f_0)_i)| +C(|a-1/2|+|c-1/2|+|b|)|\nabla (f_0)_i|\\
&\le C|D^2h|+C\|H\|_{L^2(d\mu_f)},
\end{aligned}
\]
where \(h=f-f_0\).  
Consequently,
\[
\bigl\|
\ast(d(e_1-e_1^0)\wedge de_2)
+\ast(de_1^0\wedge d(e_2-e_2^0))
\bigr\|_{\mathcal{H}^1}
\le C\|h\|_{W^{2,2}}+C\|H\|_{L^2(d\mu_f)}.
\]

Therefore
\[
\|\Delta_{g_\tau}u\|_{\mathcal{H}^1} \le C\|h\|_{W^{2,2}}+C\|H\|_{L^2(d\mu_f)}.
\]
Applying the elliptic estimate for \(\Delta_{g_\tau}\) with Hardy space data  \cite{CLMS-1993,MS-1995} yields
\[
\|u\|_{L^\infty}
+\|u\|_{W^{1,2}}
+\|u\|_{W^{2,1}}
\le C\|h\|_{W^{2,2}}+C\|H\|_{L^2(d\mu_f)}.
\]

Finally, since \(\|h\|_{W^{2,2}} \le C\|H\|_{L^2(d\mu_f)}\), we conclude
\[
\|u\|_{L^\infty}
+\|u\|_{W^{1,2}}
+\|u\|_{W^{2,1}}
\le C\|H\|_{L^2(d\mu_f)}.
\]

\end{proof}

\appendix

\section{Proof of Proposition~\ref{away from boundary of conformal group}}

In this appendix we provide the proof of Proposition~\ref{away from boundary of conformal group}.  
By the $2\pi^2$-Theorem of Marques and Neves (Theorem~\ref{2pi2-theorem}), it is enough to obtain a uniform area bound for the canonical family when the conformal parameter is close to the boundary of the unit ball.  
More precisely, Proposition~\ref{away from boundary of conformal group} will follow once we prove the following result.

\begin{theorem}[Uniform area estimate near the boundary]\label{uniformly area estimate}
There exist constants $\gamma_1>0$ and $\eta_1>0$ with the following property.  
Let $\Sigma\subset \mathbb{S}^3$ be a surface satisfying the standing hypotheses of Section~3 (in particular, $\Sigma$ is $(\gamma,r_0)$-regular with $\gamma\le\gamma_1$ and has positive genus), and assume in addition that
\[
\mathcal{W}(\Sigma)\le 8\pi.
\]
Then, for all $(v,t)\in \mathring{B}^4\times[-\pi,\pi]$ with
\[
|v|\ge 1-\eta_1,
\]
one has
\[
\mathcal{H}^2(\Sigma_{(v,t)}) \le 5\pi.
\]
\end{theorem}

\begin{remark}
Neither the constant $8\pi$ in the hypothesis nor the constant $5\pi$ in the
conclusion is essential.  
The upper bound $8\pi$ may be replaced by any fixed finite constant (with the
corresponding $\eta_{1}$ depending on it), and the number $5\pi$ may be replaced
by any constant strictly larger than $4\pi$.  
The specific values are chosen only for convenience and to streamline the
presentation.
\end{remark}

The remainder of this appendix is devoted to the proof of Theorem~\ref{uniformly area estimate}, and hence of Proposition~\ref{away from boundary of conformal group}. For clarity, the proof is divided into three subsections.

 \subsection{Estimate the area for $|t|$ close to $0$ }

 \begin{lemma}\label{small t} 
There exists a constant $C>0$ such that for $\gamma\ll 1$, there exists $\eta(\gamma)>0$ such that for any $v\in \mathbb{R}^4$ with $1-|v|\le \eta(\gamma)$ and $t\in [-\pi,\pi]$, if $\Sigma\subset \mathbb{S}^3$ is an $(\gamma,r_0)$-regular torus,  then:
\begin{align*}
\mathrm{Area}(\Sigma_{(v,t)}) \le 4\pi(1+C\gamma^{\frac{1}{2}}) + 16\mathcal{W}(\Sigma) |\sin t|.
\end{align*}
\end{lemma}

\begin{proof}
By definition, $\Sigma_{(v,t)} \subset P_{v,t}(\Sigma)$, where
$$P_{v,t}(x) = \cos t \left((1-|v|^2)\frac{x-v}{|x-v|^2}-v\right) + \sin t\left(N(x) - \frac{2\langle N(x), x-v\rangle}{|x-v|^2}(x-v)\right).$$
Thus
\begin{align}\label{area calculation}
\mathrm{Area}(\Sigma_{(v,t)}) \le \int_{\Sigma} J_{P_{v,t}}(x) d\mathcal{H}^2(x).
\end{align}
For $F_v(x) = (1-|v|^2)\frac{x-v}{|x-v|^2}-v$, we have
\begin{align*}
DF_v(x) = \frac{1-|v|^2}{|x-v|^2} Q_{x,v},
\end{align*}
where $Q_{x,v}$ is the reflection
$
Q_{x,v}(w) = w - 2\frac{\langle w, x-v\rangle}{|x-v|^2}(x-v).
$
For $w \in T_x\Sigma$, by  $\langle w, N(x)\rangle = 0$, we compute
\begin{align*}
DP_{v,t}(w) &= \cos t \frac{1-|v|^2}{|x-v|^2} Q_{x,v}(w) + \sin t\left(D_wN - 2\frac{\langle x-v, D_wN\rangle}{|x-v|^2}(x-v) - 2\frac{\langle x-v, N\rangle}{|x-v|^2}Q_{x,v}(w)\right) \\
&= \cos t \frac{1-|v|^2}{|x-v|^2} Q_{x,v}(w) + \sin t\left(Q_{x,v}(D_wN) - 2\frac{\langle x-v, N\rangle}{|x-v|^2}Q_{x,v}(w)\right).
\end{align*}
Take $w = e_i$ ($i=1,2$) to be principal directions with $D_{e_i}N = \kappa_i N$. Then,
\begin{align*}
DP_{v,t}(e_i) = \left(\cos t \frac{1-|v|^2}{|x-v|^2} + \sin t\left(\kappa_i - 2\frac{\langle x-v, N\rangle}{|x-v|^2}\right)\right) Q_{x,v}(e_i).
\end{align*}
The Jacobian satisfies
\begin{align*}
J_{P_{v,t}} &\le \underbrace{\left(\frac{1-|v|^2}{|x-v|^2}\right)^2}_{=J_{P_{v,0}}} + (\sin t)^2\left(|\kappa_1\kappa_2| + 2|\kappa_1+\kappa_2|\frac{|(x-v)^\bot|}{|x-v|^2} + 4\left(\frac{|(x-v)^{\bot}|}{|x-v|^2}\right)^2\right) \\
&\quad + |\sin t|\frac{1-|v|^2}{|x-v|^2}\left(|\kappa_1+\kappa_2| + 4\frac{|(x-v)^\bot|}{|x-v|^2}\right),
\end{align*}
where $(x-v)^\bot = \langle x-v, x\rangle x + \langle x-v, N\rangle N$ satisfies $|(x-v)^\bot| \ge |\langle x-v, N\rangle|$.

Let $\vec{H}$ and $\vec{A}$ be the mean curvature vector and second fundamental form of $\Sigma\subset \mathbb{R}^4$. Since $\vec{H} = HN + 2x$ and $A = \langle \vec{A}, N\rangle$, we have
\begin{align*}
|\kappa_1+\kappa_2| &= |H| = |\langle \vec{H}, N\rangle| \le |\vec{H}|, \\
|\vec{H}|^2 &= |H|^2 + 4, \\
|\kappa_1\kappa_2| &\le \frac{|\langle \vec{A}, N\rangle|^2}{2} \le \frac{|\vec{A}|^2}{2}.
\end{align*}
Therefore,
\begin{align}\label{Jacobian estimate}
J_{P_{v,t}} &\le J_{P_{v,0}} + (\sin t)^2\left(\frac{|\vec{A}|^2}{2} + 2|H|\left|\frac{(x-v)^\bot}{|x-v|^2}\right| + 4\left|\frac{(x-v)^\bot}{|x-v|^2}\right|^2\right) \\
&\quad + |\sin t|\sqrt{J_{P_{v,0}}}\left(|\vec{H}| + 4\left|\frac{(x-v)^\bot}{|x-v|^2}\right|\right).
\end{align}
For a torus $\Sigma$, we have
\begin{align}\label{second fundamental form calculation}
\int_{\Sigma} |\vec{A}|^2 d\mathcal{H}^2 = \int_{\Sigma} |\vec{H}|^2 d\mathcal{H}^2 = \int_{\Sigma} (|H|^2 + 4) d\mathcal{H}^2 = 4\mathcal{W}(\Sigma).
\end{align}
By the monotonicity identity and noting $v \notin \Sigma$, we know
\begin{align*}
\int_{\Sigma} |\vec{H}|^2 d\mathcal{H}^2 = 16\int_{\Sigma} \left|\frac{\vec{H}}{4} + \frac{(x-v)^\bot}{|x-v|^2}\right|^2 d\mathcal{H}^2,
\end{align*}
which implies
\begin{align}\label{normal term calculation}
\int_{\Sigma} \left|\frac{(x-v)^\bot}{|x-v|^2}\right|^2 d\mathcal{H}^2 \le \frac{1}{4}\int_{\Sigma} |\vec{H}|^2 d\mathcal{H}^2 = \mathcal{W}(\Sigma).
\end{align}

Substituting \eqref{Jacobian estimate} into \eqref{area calculation} and applying \eqref{second fundamental form calculation} and \eqref{normal term calculation}, we get
\begin{align*}
\mathrm{Area}(\Sigma_{(v,t)}) &\le \mathrm{Area}(\Sigma_{v}) + (\sin t)^2\left(6\mathcal{W}(\Sigma) + 2\int_{\Sigma} |H|\left|\frac{(x-v)^\bot}{|x-v|^2}\right| d\mathcal{H}^2\right) \\
&\quad + |\sin t| \int_{\Sigma} \sqrt{J_{P_{v,0}}}\left(|\vec{H}| + 4\left|\frac{(x-v)^\bot}{|x-v|^2}\right|\right) d\mathcal{H}^2.
\end{align*}

Using Cauchy-Schwarz and the estimates above, then
\begin{align*}
\mathrm{Area}(\Sigma_{(v,t)}) &\le \mathrm{Area}(\Sigma_{v}) + 10(\sin t)^2\mathcal{W}(\Sigma) + 6|\sin t| \sqrt{\mathrm{Area}(\Sigma_{v})\mathcal{W}(\Sigma)} \\
&\le \mathrm{Area}(\Sigma_{v}) + 16|\sin t|\mathcal{W}(\Sigma).
\end{align*}
 Finally, by the following  Lemma \ref{conformal boundary}, we know  $\mathrm{Area}(\Sigma_{v}) \le 4\pi(1 + C\gamma^{1/2})$ for $|v|\ge 1-\eta$, and the proof is complete.
\end{proof}

\begin{lemma}\label{conformal boundary} 
There exists a constant $C>0$ such that for any $\gamma>0$, there exists $\eta(\gamma)>0$ such that for any $v\in \mathbb{R}^4$ with $0<1-|v|\le \eta(\gamma)$, if $\Sigma\subset \mathbb{S}^3$ is $(\gamma,r_0)$-regular and $\mathcal{W}(\Sigma)\le 8\pi$, then there exists a geodesic sphere $S\subset \mathbb{S}^3$ such that 
\begin{align*}
\mathrm{Area}(\Sigma_v) \le 4\pi(1+C\gamma^{\frac{1}{2}}) \quad \text{and} \quad d_{\mathcal{H}}(\Sigma_v,S) \le C\gamma^{\frac{1}{4}}.
\end{align*}
\end{lemma}

\begin{proof}
The conformal transformation decomposes as $F_v = G_v^{-1} \circ S_\lambda \circ G_v$, where
\begin{align*}
G_v(x) &= \frac{2}{1+\langle x,\frac{v}{|v|}\rangle}(x-\langle x,\tfrac{v}{|v|}\rangle \tfrac{v}{|v|}), \\
S_\lambda(y) &= \lambda y \quad \text{with} \quad \lambda = \tfrac{1+|v|}{1-|v|}, \\
G_v^{-1}(z) &= \tfrac{4z}{4+|z|^2} + \tfrac{4-|z|^2}{4+|z|^2}\tfrac{v}{|v|}.
\end{align*}
Here $G_v: \mathbb{S}^3 \to H_v$ is stereographic projection to the hyperplane orthogonal to $v$.
The pullback metric is calculated as
\begin{align*}
g_v = \frac{16}{(4+|y|^2)^2}g_{\mathbb{R}^4}|_{H_v}.
\end{align*}
Let $\tilde{\Sigma}_v = G_v(\Sigma_v) = \lambda G_v(\Sigma)$. The area splits as
\begin{align}\label{area split}
\mathrm{Area}(\Sigma_v) = \mathrm{Area}_{g_v}(\tilde{\Sigma}_v \cap B_L(0)) + \mathrm{Area}_{g_v}(\tilde{\Sigma}_v \setminus B_L(0)).
\end{align}
Choose $L = \sqrt{\tfrac{32\pi}{\gamma^{1/2}}-4}$. For the exterior part, note that for $z \in \tilde{\Sigma}_v \setminus B_L(0)$, we have
\begin{align*}
w = G_v^{-1}(z) = \tfrac{4z}{4+|z|^2} + \tfrac{4-|z|^2}{4+|z|^2}\tfrac{v}{|v|},
\end{align*}
so,
\begin{align*}
\left|w + \tfrac{v}{|v|}\right| = \tfrac{4}{\sqrt{4+|z|^2}} \le \tfrac{4}{\sqrt{4+L^2}}.
\end{align*}
Thus
\begin{align*}
\mathrm{Area}_{g_v}(\tilde{\Sigma}_v \setminus B_L(0)) = \mathrm{Area}(\Sigma_v \cap B_r(-\tfrac{v}{|v|})) \quad \text{with} \quad r = \tfrac{4}{\sqrt{4+L^2}}=\frac{\gamma^{\frac{1}{4}}}{\sqrt{2\pi}}.
\end{align*}
By the monotonicity formula in $\mathbb{R}^4$, we know
\begin{align*}
\frac{\mathrm{Area}(\Sigma_v \cap B_r(p))}{r^2} \le \frac{1}{16}\int_{\Sigma_v} |\vec{H}|^2 d\mathcal{H}^2 = \frac{1}{4}\mathcal{W}(\Sigma) \le 2\pi.
\end{align*}
Therefore,
\begin{align}\label{area outside ball}
\mathrm{Area}_{g_v}(\tilde{\Sigma}_v \setminus B_L(0)) \le 2\pi \cdot \frac{16}{4+L^2} \le  \gamma^{1/2}.
\end{align}
and 
\begin{align}\label{distance outside ball}
d_{\mathcal{H}}(\Sigma_v\backslash G_v^{-1}(B_L(0)),-\frac{v}{|v|})\le  \gamma^{\frac{1}{4}}.
\end{align}
Now we estimate the interior part $\tilde{\Sigma}_v \cap B_L(0) = \lambda(G_v(\Sigma) \cap B_{r_3}(0))$ with $r_3 = \lambda^{-1}L=\frac{(1-|v|)L(\gamma)}{1+|v|}$.

 For this, we first note that for any surface $\Gamma\subset H_v$, the area is
\begin{align*}
\text{Area}_{g_v}(\Gamma)=\int_{\Gamma}\frac{16}{(4+|z|^2)^2}d\mathcal{H}^2(z).
\end{align*}
 Without loss of generality, assume $\Gamma=G_v(\Sigma)\cap B_{r_3}(0)\neq \emptyset$ and choose  $y_0\in G_v(\Sigma)\cap B_{2r_3}(0)$ such that
 \begin{align*}
 |y_0|=\inf_{y\in G_v(\Sigma)\cap B_{2r_3}(0) }|y|.
 \end{align*}
Then $|y_0|\le r_3$ and  $B_{r_3}(0)\subset B_{2r_3}(y_0)$.

Note $DG_v^{-1}(\frac{v}{|v|})=\text{id}:T_{\frac{v}{|v|}}\mathbb{S}^3=H_v\to H_v$. There exists $r_1=r_1(\gamma)>0$ such that
\begin{align*}
(1-\gamma)|x-y|\le |G_v(x)-G_v(y)|\le (1+\gamma)|x-y|,\quad \quad  \forall x,y\in \mathbb{S}^3\cap B_{r_1}(\frac{v}{|v|}).
\end{align*}
Since $\Sigma$ is  $(\gamma,r_0)$-regular,  we know that  there exists $r_2(\gamma)<\min\{r_1(\gamma), r_0(\gamma)\}$ such that $G(\Sigma_v)$ is $(2\gamma, r_2)$-regular.

Now, choose $\eta(\gamma)$ so small enough such that $1-|v|\le \eta(\gamma)$ implies $3r_3 < r_2(\gamma)$. Then, by the definition of $(2\gamma,r_2)$-regularity,  there exists an affine plane $T_{y_0,3r_3}$ and conformal parametrization $\varphi: D_{3r_3} \to G_v(\Sigma)$ such that
\begin{align*}
|\varphi(\tau)-\tau| \le 6\gamma r_3, \quad (1-2\gamma)|\tau_1-\tau_2| \le |\varphi(\tau_1)-\varphi(\tau_2)| \le (1+2\gamma)|\tau_1-\tau_2|.
\end{align*}
Let $y_1 \in T_{y_0,3r_3} \cap B_{3r_3}(0)=D_{3r_3}$ have minimal norm.Then $$|y_1| \le |y_0| \le r_3 \quad \text{ and }\quad  |y_0-y_1|^2 + |y_1|^2 = |y_0|^2.$$

\begin{figure}[htbp]
    \centering
    \begin{subfigure}[t]{0.45\linewidth}
        \centering
        \includegraphics[width=\linewidth]{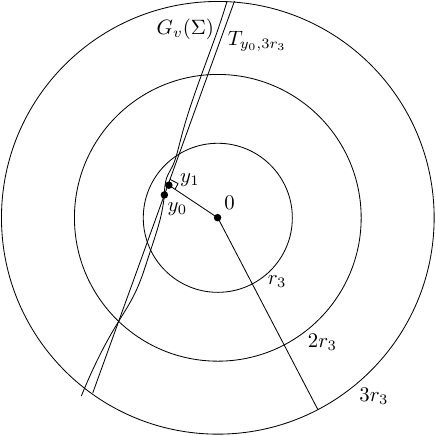}
        \caption{The stereographic configuration.}
        \label{fig:stereographic_config}
    \end{subfigure}\hfill
    \begin{subfigure}[t]{0.45\linewidth}
        \centering
        \includegraphics[width=\linewidth]{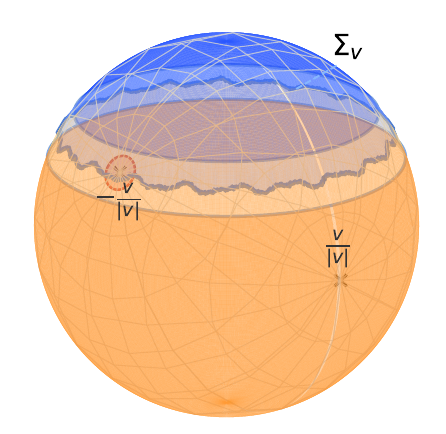}
        \caption{The configuration near $\Sigma_v$.}
        \label{fig:conformal_boundary_v}
    \end{subfigure}
    \caption{Comparison between configurations.}
    \label{fig:comparison}
\end{figure}

Since $|\varphi(y_1)-y_1| \le 6\gamma r_3$, we have $|\varphi(y_1)| \le |y_1| + 6\gamma r_3 \le (1+6\gamma)r_3$, so $\varphi(y_1) \in B_{2r_3}(0) \cap G_v(\Sigma)$. By minimality of $|y_0|$, we get $|y_0| \le |\varphi(y_1)|$, hence
\begin{align*}
|y_1|^2 + |y_0-y_1|^2= |y_0|^2 \le |\varphi(y_1)|^2 \le (|y_1| + 6\gamma r_3)^2.
\end{align*}
This implies $|y_0-y_1| \le 7\sqrt{\gamma} r_3$ and $|\varphi(y_0)-\varphi(y_1)| \le 14\sqrt{\gamma} r_3$.
Therefore,
\begin{align*}
B_{r_3}(0) \cap G_v(\Sigma) &\subset B_{2r_3}(y_0) \cap G_v(\Sigma) \\
&\subset \varphi(T_{y_0,3r_3} \cap B_{(2+2\gamma)r_3}(y_0)) \\
&\subset \varphi(T_{y_0,3r_3} \cap B_{(2+10\sqrt{\gamma})r_3}(y_1)).
\end{align*}

Let $T = \lambda T_{y_0,3r_3}$, $y_2 = \lambda y_1$, and $\tilde{\varphi}(\eta) = \lambda \varphi(\lambda^{-1}\eta)$. Then,
\begin{align*}
B_L(0) \cap \tilde{\Sigma}_v &= B_{\lambda r_3}(0) \cap \lambda G_v(\Sigma) \\
&\subset \lambda \varphi(\lambda^{-1}(T \cap B_{(2+10\sqrt{\gamma})L}(y_2))) \\
&= \tilde{\varphi}(T \cap B_{(2+10\sqrt{\gamma})L}(y_2)).
\end{align*}
Noting $L\le 20\gamma^{-\frac{1}{4}}$, we have $|D\tilde{\varphi}| \le 1+2\gamma$ and $|\tilde{\varphi}(\eta)-\eta| \le 6\gamma L\le C\gamma^{\frac{3}{4}}$. For $\eta \in B_{3L}(y_2)$,  we estimate
\begin{align*}
4+|\eta|^2 &\le 4 + (|\tilde{\varphi}(\eta)| + 6\gamma L)^2 \\
&\le 4 + |\tilde{\varphi}(\eta)|^2 + (6\gamma L)^2 + 12|\tilde{\varphi}(\eta)|\gamma L \\
&\le 4 + |\tilde{\varphi}(\eta)|^2 + C\gamma^{1/2}.
\end{align*}
Thus,
\begin{align*}
\mathrm{Area}_{g_v}(\tilde{\Sigma}_v \cap B_L(0)) 
&\le \int_{T \cap B_{3L}(y_2)} \frac{16(1+2\gamma)^2}{(4+|\tilde{\varphi}(\eta)|^2)^2} d\mathcal{H}^2(\eta) \\
&\le (1 + C\gamma^{1/2}) \int_T \frac{16}{(4+|\eta|^2)^2} d\mathcal{H}^2(\eta) \\
&\le (1 + C\gamma^{1/2}) \cdot 4\pi.
\end{align*}
Substitute this and \eqref{area outside ball} into \eqref{area split} we get $\mathrm{Area}(\Sigma_v)\le 4\pi(1+C\gamma^{\frac{1}{2}})$. 

For the Hausdorff distance estimate, note
\begin{align*}
\tilde{\Sigma}_v \cap B_L(0) \subset B_{C\gamma^{3/4}}(T), \quad T \cap B_L(0) \subset B_{C\gamma^{3/4}}(\tilde{\Sigma}_v \cap B_L(0)).
\end{align*}
Pulling back to $\mathbb{S}^3$ via the Lipschitz map $G_v^{-1}$  and combining with \eqref{distance outside ball}, gives $d_{\mathcal{H}}(\Sigma_v, S^2) \le C\gamma^{1/4}$, where $S^2=G_v^{-1}(T)$.
\end{proof}

 A direct corollary of  Lemma \ref{small t} is the following. 
\begin{corollary}\label{area estimate for small t}

There exists $t_0>0$ $\gamma_0>0$ and $\eta_0>0$ such that  for any $(\gamma,r_0)$-regular surface $\Sigma\subset \mathbb{S}^3$  with $\gamma\le \gamma_0$ and   $\mathcal{W}(\Sigma)\le 8\pi$,   if
$0<1-|v|<\eta_0$ and  either $|t|\le t_0$ or $|t\pm \pi|\le t_0$,   
then  
\begin{align*}
\text{Area}(\Sigma_{(v,t)})\le 5\pi<2\pi^2.
\end{align*}
\end{corollary}

\subsection{Estimate the area for $|t|$ away from  $0$ with $\textup{diam}(\Sigma_{(v,t)})\sim 1$}

We begin by fixing notation and describing the position of $\Sigma_{(v,t)}$.  
Let $\mathbb{S}^3\setminus \Sigma_v = A_v \cup A_v^*$ and $\mathbb{S}^3\setminus S = B \cup B^*$ denote the decompositions into connected components, and let $N_S$ be the unit normal to $S$ pointing into $B^*$.  
Define the parallel spheres
\[
S_t = \{\exp_S(tN_S(x)) = \cos t\, x + \sin t\, N_S(x) \mid x \in S\}, \qquad t \in [t_-,t_+],
\]
where $t_\pm$ are chosen so that $\mathrm{diam}(S_t)>0$ for $t\in(t_-,t_+)$ and $\mathrm{diam}(S_{t_\pm})=0$.  
For $t\in[-\pi,t_-)\cup(t_+,\pi]$, set $S_t=\emptyset$.  
Then $B^*=\bigcup_{0<t\le\pi}S_t$ and $B=\bigcup_{0<t\le\pi}S_{-t}$.

\begin{lemma}[Position of the parallel family]\label{position estimate} 
Let $\Sigma\subset\mathbb S^3$ be $(\gamma,r_0)$-regular with $\mathcal W(\Sigma)\le8\pi$.  
Then for every $\varepsilon>0$ there exist $\gamma(\varepsilon),\eta(\varepsilon)>0$ such that if $\gamma\le\gamma(\varepsilon)$ and $|v|\ge1-\eta(\varepsilon)$, then for all $t\in[-\pi,\pi]$, the parallel surface $\Sigma_{(v,t)}$ satisfies
\begin{equation}\label{bounded in two bands}
\Sigma_{(v,t)}\subset B_\varepsilon(S_t),
\end{equation}
where $S_t$ is the parallel sphere of $S=G_v^{-1}(T)$ at signed distance $t$.
\end{lemma}

\begin{proof}
In the stereographic chart $G_v$, the surface $B_L(0)\cap\tilde{\Sigma}_v$ lies in the thin tubular neighborhood $B_{C\gamma^{3/4}}(T)$ of $T$, and $(B_L(0)\cap H_v)\setminus B_{C\gamma^{3/4}}(T)$ has exactly two connected components.  
Since $\tilde{\Sigma}_v$ is $W^{2,2}$-close to $T$, it is homotopic to $T$ within this neighborhood while fixing $\partial B_L(0)$.  
Using mod-$2$ intersection numbers of transverse paths, one finds that the two components of $(B_L(0)\cap H_v)\setminus\tilde{\Sigma}_v$ correspond bijectively to those of $(B_L(0)\cap H_v)\setminus B_{C\gamma^{3/4}}(T)$.  
Under the stereographic projection $G_v^{-1}$, these local sides map to the global components $A_v$ and $A_v^*$ of $\mathbb S^3\setminus\Sigma_v$; the side containing a neighborhood of $-v/|v|$ corresponds to $A_v^*$ by the exterior control~\eqref{distance outside ball}.  
Thus the two sides of $\Sigma_v$ coincide with those of the comparison sphere $S=G_v^{-1}(T)$.

Combining this side correspondence with the Hausdorff estimate $d_{\mathcal H}(\Sigma_v,S)\le C\gamma^{1/4}$ yields the two–sided containment
\[
\{\phi_S>\varepsilon\}\subset A_v^*\subset\{\phi_S>-\varepsilon\},
\qquad
\{\phi_S<-\varepsilon\}\subset A_v\subset\{\phi_S<\varepsilon\},
\]
where $\phi_S$ is the signed distance to $S$ (positive on $B^*$) and $\varepsilon\sim C\gamma^{1/4}$.  
Since the normal flow of $S$ preserves inclusions of superlevel sets, applying it to these bands gives, for each $t\in[t_0,\pi-t_0]$,
\[
\{\phi_S>t+\varepsilon\}\subset A_{(v,t)}^*\subset\{\phi_S>t-\varepsilon\},
\]
where $A_{(v,t)}^*$ denotes the outward parallel set of $A_v^*$ with boundary $\Sigma_{(v,t)}$.  
Taking boundaries yields
\[
\Sigma_{(v,t)}\subset\{|\phi_S-t|\le\varepsilon\}=B_\varepsilon(S_t).
\]
Because $\Sigma_{(v,t)}$ and $S_t$ lie on the same $B^*$–side of $S$, while $B_\varepsilon(S_{-t})\subset B$ for $t>\varepsilon$, the inclusion~\eqref{bounded in two bands} follows.  
\end{proof}

For convenience, we introduce some notations.  For $x\in \Sigma_{v,t}$, choose $q_{-\varepsilon}\in S_{-\varepsilon}$ such that $d(x,S_{-\varepsilon})$ and let $N_(q_{-\varepsilon })\in T_{q_{-\varepsilon}}^{\bot}S_{-\varepsilon}$ be the unit normal vector. We denote 
 $$r_\varepsilon(x)=d(x,S_{-\varepsilon})-t$$
 and  define the projection $\pi_\varepsilon: \Sigma_{v,t}\to S_{t-\varepsilon}$ by   $$\pi_\varepsilon(x)=\cos t q_{-\varepsilon}+\sin t N(q_{-\varepsilon}).$$

To proceed, we need a quantitative estimate for distance-realizing points.

\begin{lemma}\label{quantitative uniqueness of distance realizing point}
For fixed $d > 0$ and $\varepsilon > 0$, define $\delta_1(\varepsilon|d)$ as the supremum of $d(p,q)$ over $p,q \in S_{-\varepsilon}$ with $\mathrm{diam}(S_t) \ge d$, $x \in B_\varepsilon(S_t)$ satisfying $d(x,q) = d(x,S_{-\varepsilon})$ and $|d(x,p) - d(x,q)| \le 2\varepsilon$. Then,
$$\lim_{\varepsilon\to 0} \delta_1(\varepsilon|d) = 0.$$
\end{lemma}

\begin{proof}
Assume otherwise. Then there exist sequences $\varepsilon_i\to 0$, $p_i,q_i\in S_{-\varepsilon_i}$ with $d(p_i,q_i)\ge\delta_0>0$, and $x_i\in B_{\varepsilon_i}(S_{t_i})$ with $\mathrm{diam}(S_{t_i})\ge d$, $d(x_i,q_i)=d(x_i,S_{-\varepsilon_i})$, and $|d(x_i,p_i)-d(x_i,q_i)|\le 2\varepsilon_i$. Passing to limits, $x_i\to x\in S_t$, $p_i\to p\in S$, $q_i\to q\in S$ with $d(x,p)=d(x,q)=d(x,S)$ and $d(p,q)\ge\delta_0$, contradicting uniqueness of distance-realizing points when $\mathrm{diam}(S_t)>0$.
\end{proof}

With this preparation, we can now estimate the angle between normals.

\begin{lemma}\label{support sphere and angle estimate} 
Let $t_0$ be as in Corollary \ref{area estimate for small t} and $d > 0$ be small. For any $\varepsilon > 0$, let $\eta(\varepsilon), \gamma(\varepsilon) > 0$ be as in Lemma \ref{position estimate}. Assume $\Sigma \subset \mathbb{S}^3$ be $(\gamma,r_0)$-regular with $\gamma \le \gamma(\varepsilon)$ and $\mathcal{W}(\Sigma) \le 8\pi$. Then for any $|t|\geq t_0$, $|v| \ge 1 - \eta(\varepsilon)$, and $x \in \Sigma_{(v,t)}$, there exists a supporting ball $B_t(p)$ with
\begin{align}\label{support sphere}
x \in \partial B_t(p) \quad \text{and} \quad \Sigma_{(v,t)} \cap \mathring{B}_t(p) = \emptyset.
\end{align}
Moreover, if $\mathrm{diam}(S_t) \ge d > 0$, then the angle between normals satisfies
\begin{align} \label{small angle estimate}
\theta = \angle(N_1(x), N_2(x)) \le \delta(\varepsilon|d,t_0),
\end{align}
where $\lim_{\varepsilon \to 0} \delta(\varepsilon|d,t_0) = 0$, $N_2(x)$ points outside $B_t(p)$ and $N_1(x)$ points into $\cup_{\tau>d(x,S)}S_\tau$.
\end{lemma}

\begin{proof}
Wolg, assume $t\geq 0$. Take $p \in \Sigma_v$ with $d(x,p) = t$ for the supporting sphere.  Then $\theta = \angle pxq_{-\varepsilon}$. Set $a = t$, $b = d(x,q_{-\varepsilon}) \in [t, t+2\varepsilon]$, $c = d(p,q_{-\varepsilon})$.

By spherical cosine law, we know 
\begin{align*}
\cos c = \cos(a-b) + \sin a \sin b (\cos \theta - 1). 
\end{align*}
 Thus
\begin{align*}
1 - \cos \theta = \frac{\cos(a-b) - \cos c}{\sin a \sin b} \le \frac{1 - \cos c}{\min \{(\sin t_0)^2,(\sin t_d)^2\}},
\end{align*}
where we use $\mathrm{diam}(S_t)\ge d$ implies $a,b\le t_d$ for some $t_d<\pi.$ 

Choose $p_{-\varepsilon} \in S_{-\varepsilon}$ with $d(p,p_{-\varepsilon})=d(p,S_{-\varepsilon}) \le 2\varepsilon$. Then $|d(x,p_{-\varepsilon}) - d(x,q_{-\varepsilon})| \le 2\varepsilon$. So by Lemma \ref{quantitative uniqueness of distance realizing point}, we know
\begin{align*}
d(p_{-\varepsilon}, q_{-\varepsilon}) \le \delta_1(\varepsilon|d),
\end{align*}
hence  $c\le d(p,p_{-\varepsilon})+d(p_{-\varepsilon},q_{-\varepsilon}) \le 2\varepsilon + \delta_1(\varepsilon|d)$.

Using $1 - \cos x \le x^2/2$, we get 
\begin{align*}
1 - \cos \theta \le \frac{(2\varepsilon + \delta_1)^2}{2\min \{(\sin t_0)^2,(\sin t_d)^2\}}.
\end{align*}
For small $\theta$,  $1 - \cos \theta \ge \theta^2/4$ implies 
\begin{align*}
\theta \le \frac{2\sqrt{2}}{\min\{\sin t_0, \sin t_d\}}(2\varepsilon + \delta_1) =: \delta(\varepsilon|d,t_0),
\end{align*}
with $\lim_{\varepsilon \to 0} \delta(\varepsilon|d,t_0) = 0$.
\end{proof}

Next, we establish a key Lipschitz-type estimate.

\begin{lemma}\label{two point cone}
Let $\{x_i\}_{i=1,2}\subset B_{\varepsilon}(S_t)\cap \Sigma_{v,t}$ with supporting balls $B_t(p_i), t\ge t_0$,  $\max\theta(x_i)\le\delta(\varepsilon|d,t_0)$ and  $\mathrm{diam}(S_t)\ge d$. Then,
$$|r_\varepsilon(x_1)-r_\varepsilon(x_2)|\le\delta_2(\varepsilon|d,t_0)d(x_1,x_2),$$
where  $\lim_{\varepsilon\to 0}\delta_2(\varepsilon|d,t_0)=0$.
\end{lemma}

\begin{proof}
Suppose the estimate fails.  
Then there exist sequences $\varepsilon_k\to 0$ and points $x_{ik}\in B_{\varepsilon_k}(S_t)$ such that
\[
|r_{\varepsilon_k}(x_{1k})-r_{\varepsilon_k}(x_{2k})|\ge \delta_0\, d(x_{1k},x_{2k})
\]
for some fixed $\delta_0>0$.  
Set $\rho_k=d(x_{1k},x_{2k})\to 0$ and rescale by $\rho_k$.  
In the limit, the supporting balls converge to supporting half-spaces $H_i$ with boundary planes $P_i=\partial H_i$.  
The contact conditions $\max\theta(x_i)\le\delta(\varepsilon_k|d,t_0)\to 0$ gives $N_1(x_i)=N_2(x_i)$, while the inequality above implies that the distance from $x_2$ to $P_1$ is at least $\delta_0$.  
Since $P_1$ and $P_2$ are parallel and both support the same limiting set, this forces $H_1=H_2$, hence $P_1=P_2$, which is a contradiction.
\end{proof}

This immediately leads to a Lipschitz estimate for the projection. 

\begin{corollary}\label{Lipschitz estimate}
Under the hypotheses of Lemma \ref{two point cone}, there holds
$$d(x_1,x_2)\le(1+\delta_3(\varepsilon|d,t_0))d(\pi_\varepsilon(x_1),\pi_\varepsilon(x_2)),$$
 where $\lim_{\varepsilon\to 0}\delta_3(\varepsilon|d,t_0)=0$.
\end{corollary}

\begin{proof}
Assume without loss of generality that $r_\varepsilon(x_2)\ge r_\varepsilon(x_1)$. Define the auxiliary point
$$\hat{x}_2 = \exp_{\pi_\varepsilon(x_2)}(r_\varepsilon(x_1)N_2(\pi_\varepsilon(x_2))).$$
Then $\hat{x}_2$ lies on the same level set as $x_1$ (i.e., $S_{t-\varepsilon+r_\varepsilon(x_1)}$), and we have
$$|r_\varepsilon(x_1)-r_\varepsilon(x_2)| = d(x_2,\hat{x}_2).$$

By Lemma \ref{two point cone}, we know 
$$d(x_2,\hat{x}_2) \le \delta_2(\varepsilon|d,t_0) d(x_1,x_2).$$

The triangle inequality implies
\begin{align*}
d(x_1,x_2) &\le d(x_1,\hat{x}_2) + d(\hat{x}_2,x_2) \\
&\le d(x_1,\hat{x}_2) + \delta_2 d(x_1,x_2).
\end{align*}
Rearranging gives
$$d(x_1,x_2) \le \frac{1}{1-\delta_2} d(x_1,\hat{x}_2).$$

Now estimate $d(x_1,\hat{x}_2)$. Both points lie on the parallel surface $S_{t-\varepsilon+r_\varepsilon(x_1)}$, while their projections $\pi_\varepsilon(x_1)$ and $\pi_\varepsilon(x_2)$ lie on $S_{t-\varepsilon}$. Since $\mathrm{diam}(S_t) \ge d$, the metric distortion is controlled:
$$d(x_1,\hat{x}_2) \le (1 + C_d \cdot 2\varepsilon) d(\pi_\varepsilon(x_1),\pi_\varepsilon(x_2)).$$

Combining these estimates:
\begin{align*}
d(x_1,x_2) &\le \frac{1 + 2C_d\varepsilon}{1 - \delta_2} d(\pi_\varepsilon(x_1),\pi_\varepsilon(x_2)) \\
&= (1 + \delta_3(\varepsilon|d,t_0)) d(\pi_\varepsilon(x_1),\pi_\varepsilon(x_2)),
\end{align*}
where $\delta_3 = \frac{1 + 2C_d\varepsilon}{1 - \delta_2} - 1 \to 0$ as $\varepsilon \to 0$.
\end{proof}

Finally, we obtain the area estimate.
\begin{corollary} \label{area estimate for middle t}
Assume $\Sigma \subset \mathbb{S}^3$ be $(\gamma,r_0)$-regular with $\gamma \le \gamma(\varepsilon)$ and $\mathcal{W}(\Sigma) \le 8\pi$. Then for $|t|\geq t_0$ and $|v|\ge 1-\eta(\varepsilon)$, $\Sigma_{(v,t)}$ is a Lipschitz graph over $S_{t-\varepsilon}$ with Lipschitz constant $\le 1+\delta_3$ and:
$$\mathcal{H}^2(\Sigma_{(v,t)})\le 4\pi(1+\delta_3)^2.$$
\end{corollary}

\begin{proof}
The projection $\pi_\varepsilon:\Sigma_{(v,t)}\to S_{t-\varepsilon}$ is bijective, and its inverse $\psi$ is Lipschitz with constant $\le 1+\delta_3$ by Corollary \ref{Lipschitz estimate}. The area bound follows.
\end{proof}

\subsection{Estimate the area for $|t|$ away from  $0$  with $\textup{diam}(\Sigma_{(v,t)})\ll 1$.}

We need the following area comparison theorem, whose proof we include for completeness.
\begin{lemma} \label{area comparison}
Let $\Sigma$ be a surface in $\mathbb{S}^3$. Then for any $0<\tau< t\le \pi$ and $|v|<1$:
$$\mathcal{H}^2(\Sigma_{(v,t)})\le\left(\frac{\sin t}{\sin \tau}\right)^2\mathcal{H}^2(\Sigma_{(v,\tau)}).$$
\end{lemma}

\begin{proof}
The key idea is to compare the area evolution along normal geodesics. For $x \in \Sigma_{(v,t)}$ and $p \in \Sigma_v$ with $d(x,p) = d(x, \Sigma_v)$, let $\gamma_p : [0,t] \to \mathbb{S}^3$ be the minimizing geodesic. For $\tau \in [0,t)$, we have $\gamma_p(\tau) \in \Sigma_{(v,\tau)}$ and $p \in \Sigma_v^*(\tau)$.

The mapping $\varphi_\tau: \Sigma_v^*(t) \to \Sigma_{(v,\tau)}$ is clearly injective. Consider the area functional:
$$A(\tau) = \int_{\Sigma_v^*(t)} J_{\varphi_\tau}(p) d\mathcal{H}^2.$$

By the first variation formula, we know 
$$\frac{dA(\tau)}{d\tau} = -\int_{\Sigma_v^*(t)} N(\varphi_\tau(p)) \cdot H(\varphi_\tau(p)) J_{\varphi_\tau}(p)d\mathcal{H}^2.$$

The curvature estimate comes from a geometric comparison argument. For $q = \varphi_\tau(p)$, the distance function $d_p$ attains its minimum along $\Sigma_{(v,\tau)}$ at $q$, since $d_p(q) = \tau$ and $d_p(y) \ge \tau$ for nearby $y \in \Sigma_{(v,\tau)}$. This implies:
$$\langle N(\varphi_\tau(p)), H(\varphi_\tau(p))\rangle \ge -2\cot\tau.$$

Substituting into the variation formula, we get
$$\frac{dA(\tau)}{d\tau} \le 2\cot\tau A(\tau).$$

Integrating from $\tau$ to $t$ gives
$$A(t) \le \left(\frac{\sin t}{\sin \tau}\right)^2 A(\tau).$$

Since $A(t) = \mathcal{H}^2(\Sigma_{(v,t)})$ and $A(\tau) \le \mathcal{H}^2(\Sigma_{(v,\tau)})$, the result follows.
\end{proof}

We now combine these ingredients to complete the proof of
Theorem~\ref{uniformly area estimate}.

\begin{proof}[Proof of Theorem~\ref{uniformly area estimate}]
By Corollary~\ref{area estimate for small t} there exist $\gamma_0,t_0,\eta_0>0$ such that, whenever $\gamma\le \gamma_0$ and $|v|\ge 1-\eta_0$,
\[
\mathcal{H}^2(\Sigma_{(v,t)}) \le 5\pi \qquad \text{for all } |t|\le t_0.
\]
Fix $\varepsilon>0$ small. By Lemma~\ref{position estimate}, there exist $\gamma(\varepsilon)\in(0,\gamma_0]$ and $\eta(\varepsilon)\in(0,\eta_0]$ such that if $\gamma\le \gamma(\varepsilon)$ and $|v|\ge 1-\eta(\varepsilon)$, then
\[
\Sigma_{(v,t)}\subset B_\varepsilon(S_t)\qquad\text{for all } t\in(-\pi,\pi),
\]
where $B_\varepsilon(S_t)$ denotes the ambient $\varepsilon$-tubular neighborhood of the geodesic sphere $S_t\subset\mathbb S^3$.

By symmetry it suffices to treat $t\ge0$. Set
\[
d:=\sin t_0/2\in(0,1).
\]

Let $t_1\in[0,\pi)$ be the maximal number such that $\mathrm{diam}(S_t)\ge d$. By continuity, $\mathrm{diam}(S_{t_1})=d$. Choose $\tau>0$ so that
\[
\mathrm{diam}(S_{t_1+\tau})=\frac{d}{2}.
\]
Since on $[t_1,\pi)$ the function $t\mapsto \mathrm{diam}(S_t)$ is strictly decreasing and and in fact equals $\sin (t_0/2+t_1-t)$, there exists a universal $c\in(0,1)$ (e.g. $c=1/10$ suffices) such that
\begin{equation}\label{eq:tau-lower}
\tau\ge c\,t_0.
\end{equation}

Applying Corollary~\ref{area estimate for middle t} (with the localization from Lemma~\ref{position estimate}), we obtain
\begin{equation}\label{eq:middle-bound}
\mathcal{H}^2(\Sigma_{(v,t)}) \le \max\!\bigl\{4\pi\,(1+\delta_3(\varepsilon\mid t_0))^2,\,5\pi\bigr\}
\quad\text{for all } t\in[0,\max\{t_0,t_1+\tau\}],
\end{equation}
where $\delta_3(\varepsilon\mid t_0)\to0$ as $\varepsilon\to0$ with $t_0$ fixed.

For $t\in[t_1+\tau,\pi)$ we invoke Lemma~\ref{area comparison} between the levels $t$ and $t_1+\tau$ to get
\begin{equation}\label{eq:area-comp}
\mathcal{H}^2(\Sigma_{(v,t)})
\;\le\; \Bigl(\frac{\sin(t-t_1)}{\sin \tau}\Bigr)^2
\mathcal{H}^2(\Sigma_{v,t_1+\tau}).
\end{equation}
Because $t-t_1\le t_0/2+\varepsilon$ we deduce from \eqref{eq:tau-lower} that
\begin{equation}\label{eq:ratio-bound}
\Bigl(\frac{\sin(t-t_1)}{\sin \tau}\Bigr)^2
\;\le\; C_1\qquad\text{with } C_1\leq \Bigl(10\frac{t_0/2+\varepsilon}{t_0}\Bigr)^2\le 100
\end{equation}
once $\varepsilon\le t_0/2$. Moreover, applying Corollary~\ref{area estimate for middle t} at the time $t_1+\tau$ and using $\sin(t_1+\tau)\le \sin t_0/2=d\le t_0/2$, we have
\begin{equation}\label{eq:base-level}
\mathcal{H}^2(\Sigma_{v,t_1+\tau})
\;\le\; 4\pi\,(1+\delta_3(\varepsilon\mid t_0))^2\,\sin^2(t_1+\tau)
\;\le\; \pi\,(1+\delta_3(\varepsilon\mid t_0))^2\,t_0^2.
\end{equation}
Combining \eqref{eq:area-comp}, \eqref{eq:ratio-bound} and \eqref{eq:base-level} yields
\[
\mathcal{H}^2(\Sigma_{(v,t)})
\;\le\; C_1\pi\,(1+\delta_3(\varepsilon\mid t_0))^2\,t_0^2
\qquad\text{for all } t\in[t_1+\tau,\pi).
\]

Choose $t_0>0$ so small that $C_1\pi\,t_0^2\le 2\pi$ (for instance, any $t_0$ with $t_0^2\le \tfrac{1}{2C_1}$ works), and then take $\varepsilon>0$ so small that $\delta_3(\varepsilon\mid t_0)\le 10^{-2}$. With these choices we obtain
\[
\mathcal{H}^2(\Sigma_{(v,t)}) \le 5\pi
\qquad\text{for all } t\in[0,\pi).
\]
The case $t\le0$ is identical by symmetry. Finally, set $\gamma_1:=\gamma(\varepsilon)$ and $\eta_1:=\eta(\varepsilon)$ from Lemma~\ref{position estimate}. This completes the proof.
\end{proof}

\section{Global Conformal Parametrization and Smooth Approximation}

In this section we globalize the local conformal parametrization
Theorem \cite[Thm.~1.1]{BZ-2022b} and prove that any compactly supported unit-density
integral $2$-varifold with $L^2$-mean curvature admits a global conformal
parametrization by a compact smooth surface, and can be approximated in
$W^{2,2}$ by smooth embedded surfaces.  A simple but useful observation is
that, near a good rigid approximation, one can smooth by convolution while
retaining quantitative $W^{1,\infty}$-control, as recorded in the following
local lemma.

\begin{lemma}[Local mollification near a rigid map]
\label{lem:local-mollification}
Let $B_{2\varepsilon}(x_0)\subset\mathbb{R}^2$, let
$f\in W^{2,2}(B_{2\varepsilon}(x_0),\mathbb{R}^n)$, and assume that there
exists a rigid isometry $P_{x_0,2\varepsilon}:\mathbb{R}^2\to\mathbb{R}^n$
such that
\[
\varepsilon^{-1}\|f-P_{x_0,2\varepsilon}\|_{L^2(B_{2\varepsilon}(x_0))}
 +\|\nabla^2(f-P_{x_0,2\varepsilon})\|_{L^2(B_{2\varepsilon}(x_0))}\le \eta.
\]
Let $\rho_\varepsilon$ be a standard mollifier and set
$f_\varepsilon:=f*\rho_\varepsilon$ on $B_{\varepsilon}(x_0)$.  Then
\begin{equation}\label{eq:moll-W1inf}
\varepsilon^{-1}\|f_\varepsilon-P_{x_0,2\varepsilon}\|_{L^\infty(B_{\varepsilon}(x_0))}
+\|\nabla(f_\varepsilon-P_{x_0,2\varepsilon})\|_{L^\infty(B_{\varepsilon}(x_0))}
 \le C\,\eta,
\end{equation}
where $C$ depends only on the mollifier kernel.
\end{lemma}

\begin{proof}
Set $g:=f-P_{x_0,2\varepsilon}$.  By translation and scaling we may assume
$x_0=0$, $\varepsilon=1$, so we work on $B_2$ and
\[
\|g\|_{L^2(B_2)}+\|\nabla^2 g\|_{L^2(B_2)}\le C\,\eta.
\]
A standard Gagliardo-Nirenberg interpolation inequality yields
\[
\|\nabla g\|_{L^2(B_2)}
 \le C\,\|g\|_{L^2(B_2)}^{1/2}\,\|\nabla^2 g\|_{L^2(B_2)}^{1/2}
 \le C\,\eta.
\]

Let $\rho$ be a fixed unit-scale mollifier and define $g_1:=g*\rho$ on $B_1$.
Young‘s inequality for convolutions gives
\[
\|g_1\|_{L^\infty(B_1)}
 \le \|g\|_{L^2(B_2)}\,\|\rho\|_{L^2}\le C\,\eta,\qquad
\|\nabla g_1\|_{L^\infty(B_1)}
 \le \|\nabla g\|_{L^2(B_2)}\,\|\nabla\rho\|_{L^2}\le C\,\eta.
\]
Undoing the scaling back to $B_\varepsilon(x_0)$, and using that
$g_\varepsilon:=g*\rho_\varepsilon$ is the rescaled version of $g_1$, we
obtain
\[
\varepsilon^{-1}\|g_\varepsilon\|_{L^\infty(B_\varepsilon(x_0))}
+\|\nabla g_\varepsilon\|_{L^\infty(B_\varepsilon(x_0))}
 \le C\,\eta.
\]
Since $g_\varepsilon=f_\varepsilon-P_{x_0,2\varepsilon}$, this is exactly
\eqref{eq:moll-W1inf}.
\end{proof}

The next step is to combine the above mollification lemma with the local
conformal parametrization theory of \cite{BZ-2022b}.  We now record the resulting unified local $W^{2,2}$ regularity and rigid approximation statement.

\begin{theorem}[Unified local $W^{2,2}$ regularity and rigid approximation]
\label{thm:local-rigidity}
Let $V=\underline{v}(\Sigma,1)$ be an integral $2$-varifold in 
$U\supset B_1\subset\mathbb{R}^n$ with unit density and generalized mean 
curvature $H\in L^2(d\mu)$, and assume $0\in\Sigma=\operatorname{supp}V$.
Suppose that for some sufficiently small $\varepsilon>0$,
\begin{equation}\label{eq:density-small}
\frac{\mu(B_1)}{\pi}\le 1+\varepsilon,
\qquad
\int_{B_1}|H|^2\,d\mu\le\varepsilon .
\end{equation}
Then there exist:
\begin{itemize}
\item a bi-Lipschitz conformal parametrization 
      \( f:D_1\to\Sigma \) with \(f(0)=0\), satisfying 
      \( \Sigma\cap B(0,1-\psi)\subset f(D_1)\);
\item a function \( w\in W^{1,2}(D_1)\cap C^0(D_1) \) with
\[
df\otimes df=e^{2w}g_{\mathbb{R}^2},
\qquad
\|w\|_{C^0(D_1)}
+\|\nabla w\|_{L^2(D_1)}
+\|\nabla^2 w\|_{L^1(D_1)}
\le \psi(\varepsilon),
\]
\end{itemize}
where $\psi(\varepsilon)\to0$ as $\varepsilon\to0$.

Moreover, the following quantitative affine approximation holds at every
centre and scale:

\medskip
\noindent\textbf{(Local rigid approximation).}
For every $D(x,r)\subset D_1$ there exists a rigid affine map
\[
P_{x,r}(y)=T_{x,r}(y-x)+b_{x,r},
\]
with $T_{x,r}:\mathbb{R}^2\to\mathbb{R}^n$ a linear isometric embedding, such
that
\begin{equation}\label{eq:rigid-W22}
\sup_{D(x,r)}|f-P_{x,r}|
\;+\;
r\,\|\nabla(f-P_{x,r})\|_{L^\infty(D(x,r))}
\;+\;
r\|\nabla^2 f\|_{L^2(D(x,r))}
\;\le\; \psi(\varepsilon)\, r .
\end{equation}

In particular, $f$ is $(1+\psi(\varepsilon))$–bi-Lipschitz and
\[
\|\nabla^2 f\|_{L^2(D_1)}\le \psi(\varepsilon).
\]
\end{theorem}

\begin{remark}
The theorem is a direct combination of 
\cite[Thm.~1.1]{BZ-2022b} and \cite[Thm.~5.3]{BZ-2022b}.  
The local $W^{2,2}$-regularity, the existence of a conformal parametrization,
and the control of the conformal factor $w$ come from 
\cite[Thm.~1.1]{BZ-2022b}.  
The quantitative affine approximation at every centre and scale is provided by
\cite[Thm.~5.3]{BZ-2022b}.  
Because \cite[Thm.~1.1]{BZ-2022b} gives $w$ with 
$\|w\|_{L^\infty}+\|\nabla w\|_{L^2}$ arbitrarily small when $\varepsilon$ is
small, the scaling coefficient $a_{x,r}$ in \cite[Thm.~5.3]{BZ-2022b} can be
chosen to be $1$.

A detailed argument in the case where the domain is $\mathbb{R}^2$ appears
in \cite[Thm.~3.1]{BZ25}; the modification to discs is entirely routine.
\end{remark}

We now globalize the local quantitative
Theorem~\ref{thm:local-rigidity}.  
By patching the local conformal charts and using the mollification stability
from Lemma~\ref{lem:local-mollification}, one obtains a global conformal
parametrization of the underlying varifold support, and moreover a
smoothing procedure that preserves both the $W^{2,2}$-control
and the near-rigidity at all scales.  
This leads to the following global approximation result.

\begin{theorem}[Approximation by smooth embeddings]\label{thm:smooth-approx}
Let $V=\underline{v}(\Sigma,1)$ be a compactly supported integral $2$-varifold
in $\mathbb{R}^n$ with unit density and generalized mean curvature
$H\in L^2(d\mu)$.  
Then there exist a compact smooth surface $M$ and a conformal map 
\[
F\in W^{2,2}(M,\mathbb{R}^n)\cap C^{0,1}(M),
\qquad 
F_\#(\mathcal{H}^2\llcorner M)=\mu,\quad \underline{v}(F(M),1)=V.
\]

Moreover, for every $\eta>0$ small enough, there exists a smooth embedding
\[
F_\eta\in C^\infty(M,\mathbb{R}^n)
\]
such that
\[
\|F_\eta-F\|_{W^{2,2}(M)}\le \eta,
\qquad
\sup_M |F_\eta-F|\le \eta,
\]
and for all $x,y\in M$,
\begin{equation}\label{eq:bilip-Fk-vs-F}
(1-\eta)\,|F(x)-F(y)|
 \ \le\ |F_\eta(x)-F_\eta(y)|
 \ \le\ (1+\eta)\,|F(x)-F(y)|.
\end{equation}
\end{theorem}

\begin{remark}
(Properness).
The assumption that $\spt V$ is compact may be replaced by the requirement that 
$\Sigma=\spt\mu$ is \emph{proper} in $\mathbb{R}^n$.  
The proof then proceeds on $\Sigma\cap B_R$ and passes to the limit $R\to\infty$ by exhaustion.
\end{remark}

\begin{remark}
(Density).
The hypothesis $\Theta^2(\mu,x)=1$ can be replaced by 
$\Theta^2(\mu,x)<\tfrac32$ for all $x$.  
By the Allard-Almgren\cite{AA-1976} classification of one-dimensional stationary varifolds, this already forces the density to be identically~$1$. See \cite[Theorem 5.2]{BZ25} and \cite[Theorem 1.1]{RS25}. 
\end{remark}

\begin{remark}
($W^{2,2}$-conformal parametrizations).
The smooth approximations in the theorem may be refined to smooth 
approximations in the \emph{$W^{2,2}$-conformal sense}: the maps $F_\eta$
converge to $F$ in $W^{2,2}$, their induced conformal factors converge in 
$L^\infty$, and their conformal structures converge as well.  

Since $F_\eta^{*}g_{\mathbb{R}^n}\to F^{*}g_{\mathbb{R}^n}$ in $L^\infty$, their
Teichm\"uller distance (see Tromba~\cite{T92} or Fletcher-Markovic~\cite{FM07})
tends to $0$.  
Let $g$ be the constant-curvature metric conformal to $F^{*}g_{\mathbb{R}^n}$.
For small $\eta$, there exist constant-curvature metrics $g_\eta$ on $M$,
smoothly converging to $g$ in moduli space, and diffeomorphisms
$\varphi_\eta:M\to M$ satisfying
\[
\varphi_\eta^{*}(F_\eta^{*}g_{\mathbb{R}^n})=g_\eta,
\qquad
\varphi_\eta\to\mathrm{id}\ \text{uniformly},
\]
the uniform convergence following from the fact that
$\varphi_\eta$ is $(1+\psi(\eta))$-quasiconformal as a map
$(M,g)\to(M,g)$.

Set $G_\eta:=F_\eta\circ\varphi_\eta$.  
Then $G_\eta^{*}g_{\mathbb{R}^n}$ is conformal to $g_\eta$ and
$G_\eta\to F$ in $C^0$.
To upgrade the convergence of conformal factors, fix any small radius $r>0$.
Because the construction of $F_\eta$ gives

Because the construction of $F_\eta$ gives
\[
\limsup_{\eta\to0}\int_{B_r(x)} |H_{G_\eta}|^2 \le \delta(r),
\qquad \delta(r)\to0 \text{ as } r\to0,
\]
the same local argument
as in Lemma~\ref{local estimate} applies and yields
\[
\|u_{G_\eta}-u_F\|_{L^\infty(B_{r/2}(x))}\leq \psi(\delta(r)),
\]
where $\psi(\delta)\to0$ as $\delta\to0$.
Since $r$ is arbitrary, the conformal factors converge in $L^\infty$.
Finally, using the mean curvature equation upgrades
$G_\eta\to F$ to strong convergence in $W^{2,2}$.
\end{remark}

\begin{proof}
Fix $\delta>0$ small.  By the monotonicity formula and $H\in L^2(d\mu)$, for
each $x\in\Sigma$ there exists $r_0=r_0(\delta)>0$ such that
\[
\frac{\mu(B_{2r_0}(x))}{\pi(2r_0)^2}\le 1+\varepsilon(\delta),\qquad
\int_{B_{2r_0}(x)}|H|^2\le\varepsilon(\delta),
\]
with $\varepsilon(\delta)\to0$ as $\delta\to0$, and $r_0(\delta)$ can be
chosen uniformly by compact support.  Choose points
$x_1,\dots,x_N\in\Sigma$ such that
$\Sigma\subset\bigcup_{j=1}^N B_{r_0}(x_j)$.  Applying
Theorem \ref{thm:local-rigidity} on each $B_{2r_0}(x_j)$ (after a rigid motion flattening
the tangent plane) yields conformal bi-Lipschitz maps
\[
f_j:D_{2r_0}\to\Sigma\cap B_{2r_0}(x_j)
\]
such that
\[
\|f_j-i_j\|_{W^{2,2}(D_{2r_0})}\le \psi(\delta),
\]
where $i_j:D_{2r_0}\hookrightarrow\mathbb{R}^n$ denotes the standard isometric embedding obtained after the rigid motion centered at $x_j$, and
\[
(1-\psi(\delta))|x-y|\le |f_j(x)-f_j(y)|
\le (1+\psi(\delta))|x-y|
\quad\text{for all }x,y\in D_{2r_0}.
\]

Let $U_j:=f_j(D_{r_0})$; these sets cover $\Sigma$.  
Since each $f_j$ is conformal and $(1+\psi(\delta))$–bi-Lipschitz on
$D_{2r_0}$, whenever $U_i\cap U_k\neq\emptyset$ the set
$f_i^{-1}(U_i\cap U_k)$ lies in a fixed compact disc $K\Subset D_{2r_0}$.

On $U_i\cap U_k$, the transition map
\[
\varphi_{ki}=f_k^{-1}\circ f_i
\]
is conformal and satisfies
\[
|D\varphi_{ki}|,\,|D\varphi_{ki}^{-1}|
=1+\psi(\delta).
\]
Since conformal  maps with $|DF|\equiv1$ are rigid motions, the compactness of conformal maps implies the existence of an orthogonal affine map $R_{ki}$ with
\[
\|\varphi_{ki}-R_{ki}\|_{C^{1}(K)}\le C\,\psi(\delta).
\]

Thus the charts $(U_j,\phi_j:=f_j^{-1})$ form a smooth atlas on $\Sigma$ whose transition maps are conformal, which we denote by 
$\mathcal A$, giving $\Sigma$ a smooth conformal structure $M=(\Sigma,\mathcal{A})$.  All notions of conformal, $C^{k,\alpha}$,
and $W^{k,p}$ regularity on $M$ are understood via this atlas.

With respect to this structure, the parametrizations glue to a single global
map
\[
F:M\to\mathbb{R}^n,\qquad F=f_j\circ\phi_j\ \text{on }U_j,
\]
which is a conformal map, belongs to $W^{2,2}(M)$, and satisfies
$F_\#(\mathcal{H}^2\llcorner M)=\mu$.

To smooth $F$,  set $u_j:=F\circ\phi_j^{-1}$ so $u_j$ is simply the local representation of $F$ in the $j$-th coordinate
chart (in fact $u_j=f_j$, but the notation emphasizes that $u_j$ is viewed as
a function on the coordiante chart $\phi_j(U_j)\subset\mathbb{R}^2$).
For $0<\varepsilon<r_0/4$ define
\[
u_{j,\varepsilon}:= u_j*\rho_\varepsilon,\qquad
F_{j,\varepsilon}(p):=u_{j,\varepsilon}(\phi_j(p)),\quad p\in U_j.
\]
Then $F_{j,\varepsilon}\in C^\infty(U_j)$ and
$F_{j,\varepsilon}\to F$ in $W^{2,2}(U_j)$ and uniformly as
$\varepsilon\to0$.

By Theorem~\ref{thm:local-rigidity}, for each $y_0\in D_{r_0}$ and every
$\varepsilon>0$ small enough so that $B_{4\varepsilon}(y_0)\subset D_{2r_0}$,
there exists a rigid isometry $P_{y_0,4\varepsilon}^{(j)}$ with
\[
\varepsilon^{-1}\|u_j-P_{y_0,4\varepsilon}^{(j)}\|_{L^2(B_{4\varepsilon}(y_0))}
+\|\nabla^2(u_j-P_{y_0,4\varepsilon}^{(j)})\|_{L^2(B_{4\varepsilon}(y_0))}
\le C\,\psi(\delta).
\]

Lemma~\ref{lem:local-mollification} then yields
\[
\|u_j-P_{y_0,4\varepsilon}^{(j)}\|_{L^{\infty}(B_{2\varepsilon}(y_0))}
+\|u_{j,\varepsilon}-P_{y_0,4\varepsilon}^{(j)}\|_{L^{\infty}(B_{2\varepsilon}(y_0))}
+\varepsilon\|\nabla(u_{j,\varepsilon}-P_{y_0,4\varepsilon}^{(j)})\|_{L^{\infty}(B_{2\varepsilon}(y_0))}
\le C\,\psi(\delta)\,\varepsilon.
\]

Hence,
\[
|u_{j,\varepsilon}(x)-u_j(x)|\le C\,\psi(\delta)\,\varepsilon
\qquad\text{for all }x\in B_{2\varepsilon}(y_0).
\]
Since $y_0$ is arbitrary, we obtain
\[
\sup_{x\in D_{r_0}}|u_{j,\varepsilon}(x)-u_j(x)|
\le C\,\psi(\delta)\,\varepsilon.
\]
Consequently,
\begin{equation}\label{eq:uniform-C0}
\sup_{x\in U_j}|F_{j,\varepsilon}(x)-F(x)|
\le \psi(\delta)\,\varepsilon.
\end{equation}
after absorbing constants into $\psi(\delta)$.

Let $\{\eta_k\}$ be a smooth partition of unity on $M$ subordinate to
$\{U_k\}$ and set
\[
F_\varepsilon:=\sum_k \eta_k\,F_{k,\varepsilon}.
\]
Then $F_\varepsilon\in C^\infty(M)$ and $F_\varepsilon\to F$ uniformly and in
$W^{2,2}(M)$ as $\varepsilon\to0$.

Fix $p\in M$ and choose a chart $(U_j,\phi_j)$ with $p\in U_j$, writing
$x_0:=\phi_j(p)$.  
For each $k$ with $U_k\cap U_j\neq\emptyset$, the transition map
$\varphi_{kj}:=\phi_k\circ\phi_j^{-1}$ satisfies
\[
\|\varphi_{kj}-R_{kj,x_0,\varepsilon}\|_{L^\infty(B_{4\varepsilon}(x_0))}
 \;+\;
 \varepsilon\|\nabla(\varphi_{kj}-R_{kj,x_0,\varepsilon})\|_{L^\infty(B_{4\varepsilon}(x_0))}
 \;\le\; C\,\psi(\delta)\varepsilon,
\]
for an orthogonal affine map $R_{kj,x_0,\varepsilon}$ depending on $p$ and
the scale~$\varepsilon$. Such rigid maps exist for the same reason as in the earlier discussion.

Define in $j$-coordinates
\[
v_{k,\varepsilon}:=F_{k,\varepsilon}\circ\phi_j^{-1},\qquad
a_k:=\eta_k\circ\phi_j^{-1},\qquad
U_\varepsilon:=F_\varepsilon\circ\phi_j^{-1}=\sum_k a_k v_{k,\varepsilon}.
\]

Applying Lemma~\ref{lem:local-mollification} in the $k$-chart at
$y_0:=\phi_k(p)$ gives rigid isometries $P^{(k)}_{y_0,4\varepsilon}$ with
\[
\|u_{k,\varepsilon}-P^{(k)}_{y_0,4\varepsilon}\|_{L^\infty(B_{2\varepsilon}(y_0))}
+\varepsilon\|\nabla(u_{k,\varepsilon}-P^{(k)}_{y_0,4\varepsilon})\|_{L^\infty(B_{2\varepsilon}(y_0))}
\le C\,\psi(\delta)\,\varepsilon.
\]
Transport them to the $j$-chart by
\[
Q_k(x)
:=P^{(k)}_{y_0,4\varepsilon}(R_{kj,x_0,\varepsilon}(x)),
\qquad x\in\mathbb{R}^2.
\]
Then
\[
\|v_{k,\varepsilon}-Q_k\|_{L^\infty(B_{2\varepsilon}(x_0))}
+\varepsilon\|\nabla(v_{k,\varepsilon}-Q_k)\|_{L^\infty(B_{2\varepsilon}(x_0))}
\le C\,\psi(\delta)\,\varepsilon.
\]

On overlaps we have $u_j=u_k\circ\varphi_{kj}$, hence the above estimates  imply
\[
\|Q_k-Q_j\|_{L^{\infty}(B_{2\varepsilon}(x_0))}
\le C\,\psi(\delta)\,\varepsilon.
\]
Setting $P_{p,\varepsilon}:=Q_j$ yields, for all such $k$,
\[
\|v_{k,\varepsilon}-P_{p,\varepsilon}\|_{L^\infty(B_{2\varepsilon}(x_0))}
+\varepsilon\|\nabla(v_{k,\varepsilon}-P_{p,\varepsilon})\|_{L^\infty(B_{2\varepsilon}(x_0))}
\le C\,\psi(\delta)\,\varepsilon.
\]

On $B_{2\varepsilon}(x_0)\subset\phi_j(U_j)$ the overlap of the cover is
uniformly bounded, and there exists $C_1(\delta)$ such that
\[
\sum_k a_k(x)=1,\qquad
\sum_k |\nabla a_k(x)|\le \frac{C_1(\delta)}{r_0(\delta)}.
\]
Note that we have
\[
|v_{k,\varepsilon}(x)-P_{p,\varepsilon}(x)|\le C\psi(\delta)\,\varepsilon
\quad\text{for }x\in B_{2\varepsilon}(x_0),
\]
and therefore
\[
\|U_\varepsilon-P_{p,\varepsilon}\|_{L^\infty(B_{2\varepsilon}(x_0))}
 \le C\psi(\delta)\,\varepsilon.
\]
Moreover,
\[
\nabla U_\varepsilon-\nabla P_{p,\varepsilon}
 =\sum_k \nabla a_k\,(v_{k,\varepsilon}-P_{p,\varepsilon})
  +\sum_k a_k\,(\nabla v_{k,\varepsilon}-\nabla P_{p,\varepsilon}),
\]
so using the bounds on $\nabla a_k$ and the $W^{1,\infty}$-estimates for
$v_{k,\varepsilon}-P_{p,\varepsilon}$,
\[
\|\nabla U_\varepsilon-\nabla P_{p,\varepsilon}\|_{L^\infty(B_{2\varepsilon}(x_0))}
 \le C(\delta)\,\varepsilon + C\,\psi(\delta).
\]
Combining these, we have
\[
 \varepsilon^{-1}\|U_\varepsilon-P_{p,\varepsilon}\|_{L^{\infty}(B_{2\varepsilon}(x_0))}+\|\nabla(U_\varepsilon-P_{p,\varepsilon})\|_{L^{\infty}(B_{2\varepsilon}(x_0))}
 \le C(\delta)\,\varepsilon + \psi(\delta),
\]
and hence, in terms of $F_\varepsilon$ and $F$ on $U_j$,
\[
\varepsilon^{-1}|F_\varepsilon-P_{p,\varepsilon}\|_{L^{\infty}(\phi_j^{-1}(B_{2\varepsilon}(x_0)))}+
\|\nabla(F_\varepsilon-P_{p,\varepsilon})\|_{L^{\infty}(\phi_j^{-1}(B_{2\varepsilon}(x_0)))}
 \le C(\delta)\,\varepsilon + \psi(\delta).
\]

We now prove \eqref{eq:bilip-Fk-vs-F}.  
Since $F$ is $(1+\psi(\delta))$-bi-Lipschitz on each chart, there exists 
$C_0>1$ with $C_0\to1$ as $\delta\to0$ such that
\[
C_0^{-1}|\phi_j(x)-\phi_j(y)|\le |F(x)-F(y)|
\le C_0|\phi_j(x)-\phi_j(y)|,
\qquad x,y\in U_j,
\]
and for small $\delta$ we may assume $C_0\le2$.

If $|F(x)-F(y)|\le \varepsilon$, then $x,y$ lie in the same chart and
\(|\phi_j(x)-\phi_j(y)|\le2\varepsilon\), hence both belong to
$\phi_j^{-1}(B_{2\varepsilon}(x_0))$ for some $x_0$.  
Note on this ball,  we have
\[
\varepsilon^{-1}|F_\varepsilon-P_{p,\varepsilon}\|_{L^{\infty}(\phi_j^{-1}(B_{2\varepsilon}(x_0)))}+
\|\nabla(F_\varepsilon-P_{p,\varepsilon})\|_{L^{\infty}(\phi_j^{-1}(B_{2\varepsilon}(x_0)))}
 \le  \psi(\delta).
\]
As $P_{p,\varepsilon}$ is an isometry, $F_\varepsilon$ is
$(1+C\psi(\delta))$-bi-Lipschitz there, giving
\[
(1-\psi(\delta))|F(x)-F(y)|
\le |F_\varepsilon(x)-F_\varepsilon(y)|
\le (1+\psi(\delta))|F(x)-F(y)|.
\]

If $|F(x)-F(y)|\ge \varepsilon$, then by \eqref{eq:uniform-C0},
\[
|F_\varepsilon(x)-F_\varepsilon(y)|
\ge |F(x)-F(y)| - 2\psi(\delta)\varepsilon
\ge (1-\psi(\delta))|F(x)-F(y)|,
\]
since $\varepsilon\le\psi(\delta)$; the reverse inequality follows similarly.

Thus for all $x,y\in M$,
\[
(1-\psi(\delta))\,|F(x)-F(y)|
\le |F_\varepsilon(x)-F_\varepsilon(y)|
\le (1+\psi(\delta))\,|F(x)-F(y)|,
\]
which is \eqref{eq:bilip-Fk-vs-F}.

Finally, given $\eta>0$, we first choose $\delta>0$ sufficiently small so that 
$\psi(\delta)<\eta$, and then choose the mollification scale 
$\varepsilon>0$ (depending on~$\delta$) so small that
\[
\|F_\varepsilon-F\|_{W^{2,2}(M)}\le \eta,
\qquad 
\sup_M |F_\varepsilon-F|\le \eta.
\]
As shown above, this implies that for all $x,y\in M$,
\[
(1-\eta)\,|F(x)-F(y)|
 \le |F_\varepsilon(x)-F_\varepsilon(y)|
 \le (1+\eta)\,|F(x)-F(y)|.
\]
Set $F_\eta:=F_\varepsilon$.  Then $F_\eta$ is a smooth embedding,
$F_\eta\to F$ in $W^{2,2}(M)$ and uniformly, and the bi-Lipschitz estimate
holds with error~$\eta$.  

\end{proof}

\end{document}